\documentclass[reqno]{amsart}

\usepackage{tikz-cd}
\usepackage{tikz}
\usepackage{graphicx}
\usepackage{hyperref}
\usepackage{cleveref}
\usepackage{mathrsfs}
\usepackage{geometry}
\usepackage{float}
\usepackage{enumerate}
\usepackage[T1]{fontenc} 
\usepackage{tikz}
\usepackage{pgfplots}
\usepackage{amsthm} 
\usepackage{algorithm}
\usepackage{algorithmic}
\usepackage{array}
\usepackage{booktabs}

\usetikzlibrary{intersections, calc, angles, quotes}
\usetikzlibrary{arrows, positioning, shapes.geometric}

\newtheorem{theorem}{Theorem} 

\newtheorem{theoremSec}{Theorem}[section] 
\newtheorem{corollary}{Corollary} 
\newtheorem{corollarySec}{Corollary}[section] 

\newtheorem{lemma}{Lemma}[section]
\newtheorem{proposition}{Proposition}[section]

\theoremstyle{definition}
\newtheorem{definition}[lemma]{Definition}
\newtheorem{example}[lemma]{Example}

\theoremstyle{remark}
\newtheorem{remark}{Remark}[section]
\newgeometry{vmargin={22mm}, hmargin={25mm,25mm}}   

\theoremstyle{definition}

\theoremstyle{definition}

\numberwithin{algorithm}{section}
\numberwithin{equation}{section}

\newcommand{\cayleyfora}{

\pgfmathsetmacro{\h}{0.866025403784}   
\pgfmathsetmacro{\H}{3.46410161514}    

\begin{scope}
  \clip (-4,-\H) rectangle (4,\H);

  \def\M{7}

  \begin{scope}[line width=0.8pt, color=black!70]

    \foreach \i in {-\M,...,\M} {
      \foreach \j in {-\M,...,\M} {

        \pgfmathsetmacro{\Ux}{\i + 0.5*\j + 0.5}
        \pgfmathsetmacro{\Uy}{\h*(\j + 1/3)}

        \pgfmathsetmacro{\Vx}{\i + 0.5*\j + 1.0}
        \pgfmathsetmacro{\Vy}{\h*(\j + 2/3)}
        \draw (\Ux,\Uy) -- (\Vx,\Vy);

        \pgfmathsetmacro{\VxA}{\i + 0.5*(\j-1) + 1.0}
        \pgfmathsetmacro{\VyA}{\h*((\j-1) + 2/3)}
        \draw (\Ux,\Uy) -- (\VxA,\VyA);

        \pgfmathsetmacro{\VxB}{(\i-1) + 0.5*\j + 1.0}
        \pgfmathsetmacro{\VyB}{\h*(\j + 2/3)}
        \draw (\Ux,\Uy) -- (\VxB,\VyB);
      }
    }
  \end{scope}

  \begin{scope}[fill=black, draw=black]
    \foreach \i in {-\M,...,\M} {
      \foreach \j in {-\M,...,\M} {

        \pgfmathsetmacro{\Ux}{\i + 0.5*\j + 0.5}
        \pgfmathsetmacro{\Uy}{\h*(\j + 1/3)}
        \fill (\Ux,\Uy) circle (0.045);

        \pgfmathsetmacro{\Vx}{\i + 0.5*\j + 1.0}
        \pgfmathsetmacro{\Vy}{\h*(\j + 2/3)}
        \fill (\Vx,\Vy) circle (0.045);
      }
    }
  \end{scope}

\end{scope}
}

\newcommand{\basefora}{

\draw [color=black!90](0,-1.73205)--(0.5,-0.866025)--(-0.5,-0.866025)--(0,-1.73205);
\fill[fill=gray!80] (0,-1.73205)--(0.5,-0.866025)--(-0.5,-0.866025)--(0,-1.73205);

\draw[color=gray!60] (-4,3.464100)--(4,3.464100);
\draw[color=gray!60] (-4,2.59807)--(4,2.59807);
\draw[color=gray!60] (-4,1.73205)--(4,1.73205);
\draw[color=gray!60] (-4,0.866025)--(4,0.866025);
\draw[color=gray!60] (-4,0)--(4,0);
\draw[color=gray!60] (-4,-0.866025)--(4,-0.866025);
\draw[color=gray!60] (-4,-1.73205)--(4,-1.73205);
\draw[color=gray!60] (-4,-2.59807)--(4,-2.59807);
\draw[color=gray!60] (-4,-3.464100)--(4,-3.464100);

\draw[color=gray!60] (-4,3.464100)--(-4,-3.464100);
\draw[color=gray!60] (4,3.464100)--(4,-3.464100);

\draw[color=gray!60] (4,-1.73205)--(3,-3.464100);
\draw[color=gray!60] (4,0)--(2,-3.464100);
\draw[color=gray!60] (4,1.73205)--(1,-3.464100);
\draw[color=gray!60] (4,3.464100)--(0,-3.464100);
\draw[color=gray!60] (3,3.464100)--(-1,-3.464100);
\draw[color=gray!60] (2,3.464100)--(-2,-3.464100);
\draw[color=gray!60] (1,3.464100)--(-3,-3.464100);
\draw[color=gray!60] (0,3.464100)--(-4,-3.464100);
\draw[color=gray!60] (-1,3.464100)--(-4,-1.73205);
\draw[color=gray!60] (-2,3.464100)--(-4,0);
\draw[color=gray!60] (-3,3.464100)--(-4,1.73205);

\draw[color=gray!60] (-4,-1.73205)--(-3,-3.464100);
\draw[color=gray!60] (-4,0)--(-2,-3.464100);
\draw[color=gray!60] (-4,1.73205)--(-1,-3.464100);
\draw[color=gray!60] (-4,3.464100)--(0,-3.464100);
\draw[color=gray!60] (-3,3.464100)--(1,-3.464100);
\draw[color=gray!60] (-2,3.464100)--(2,-3.464100);
\draw[color=gray!60] (-1,3.464100)--(3,-3.464100);
\draw[color=gray!60] (0,3.464100)--(4,-3.464100);
\draw[color=gray!60] (1,3.464100)--(4,-1.73205);
\draw[color=gray!60] (2,3.464100)--(4,0);
\draw[color=gray!60] (3,3.464100)--(4,1.73205);

}
\newcommand{\baseforg}{
	
	\fill[fill=gray!80] (0,1.3)--(0.78,1.3)--(0,0);
	\draw [color=gray!60](0,6.5)--(0,-6.5);
	\draw [color=gray!60](-6.9,6.5)--(-6.9,-6.5);
	\draw [color=gray!60](-4.6,6.5)--(-4.6,-6.5);
	\draw [color=gray!60](-2.3,6.5)--(-2.3,-6.5);
	\draw [color=gray!60](2.3,6.5)--(2.3,-6.5);
	\draw [color=gray!60](4.6,6.5)--(4.6,-6.5);
	\draw [color=gray!60](6.9,6.5)--(6.9,-6.5);
	
	\draw [color=gray!60](-6.9,6.5)--(6.9,6.5);
	\draw [color=gray!60](-6.9,5.2)--(6.9,5.2);
	\draw [color=gray!60](-6.9,3.9)--(6.9,3.9);
	\draw [color=gray!60](-6.9,2.6)--(6.9,2.6);
	\draw [color=gray!60](-6.9,1.3)--(6.9,1.3);
	\draw [color=gray!60](-6.9,0)--(6.9,0);
	\draw [color=gray!60](-6.9,-6.5)--(6.9,-6.5);
	\draw [color=gray!60](-6.9,-5.2)--(6.9,-5.2);
	\draw [color=gray!60](-6.9,-3.9)--(6.9,-3.9);
	\draw [color=gray!60](-6.9,-2.6)--(6.9,-2.6);
	\draw [color=gray!60](-6.9,-1.3)--(6.9,-1.3);
	
	\draw [color=gray!60](-5.3,6.5)--(-6.9,3.9);
	\draw [color=gray!60](-3.8,6.5)--(-6.9,1.3);
	\draw [color=gray!60](-2.3,6.5)--(-6.9,-1.3);
	\draw [color=gray!60](-0.8,6.5)--(-6.9,-3.9);
	\draw [color=gray!60](0.8,6.5)--(-6.9,-6.5);
	\draw [color=gray!60](2.3,6.5)--(-5.3,-6.5);
	\draw [color=gray!60] (2.3,6.5) -- ($ (2.3,6.5)!0.3!(-5.3,-6.5) $);
	\draw [color=gray!60] (-5.3,-6.5) -- ($ (-5.3,-6.5)!0.6!(2.3,6.5) $);
	\draw [color=gray!60](3.8,6.5)--(-3.8,-6.5);
	\draw [color=gray!60](5.3,6.5)--(-2.3,-6.5);
	\draw [color=gray!60](6.9,6.5)--(-0.8,-6.5);
	\draw [color=gray!60](6.9,3.9)--(0.8,-6.5);
	\draw [color=gray!60](6.9,1.3)--(2.3,-6.5);
	\draw [color=gray!60](6.9,-1.3)--(3.8,-6.5);
	\draw [color=gray!60](6.9,-3.9)--(5.3,-6.5);
	
	\draw [color=gray!60](-5.3,6.5)--(2.3,-6.5);
	\draw [color=gray!60](-2.3,6.5)--(-6.9,3.9);
	
	\draw [color=gray!60](-3.8,6.5)--(3.8,-6.5);
	\draw [color=gray!60](-2.3,6.5)--(5.3,-6.5);
	
	\draw [color=gray!60](-2.3,6.5)--(6.9,1.3);
	
	\draw [color=gray!60](-0.8,6.5)--(6.9,-6.5);
	\draw [color=gray!60](0.8,6.5)--(6.9,-3.9);
	
	\draw [color=gray!60](2.3,6.5)--(-6.9,1.3);
	\draw [color=gray!60](2.3,6.5)--(6.9,-1.3);
	\draw [color=gray!60](2.3,6.5)--(6.9,3.9);
	
	\draw [color=gray!60](3.8,6.5)--(6.9,1.3);
	\draw [color=gray!60](5.3,6.5)--(6.9,3.9);
	\draw [color=gray!60](-6.9,-1.3)--(6.9,6.5);
	\draw [color=gray!60](6.9,6.5)--($ (6.9,6.5)!0.5!(-6.9,-1.3) $);
	\draw [color=gray!60](-6.9,-1.3)--($ (-6.9,-1.3)!0.33!(6.9,6.5) $);
	\draw [color=gray!60](6.9,3.9)--(-6.9,-3.9);
	\draw [color=gray!60](6.9,1.3)--(-6.9,-6.5);
	\draw [color=gray!60](6.9,-1.3)--(-2.3,-6.5);
	\draw [color=gray!60](6.9,-3.9)--(2.3,-6.5);
	
	\draw [color=gray!60](-6.9,6.5)--(0.8,-6.5);
	\draw [color=gray!60](-6.9,6.5)--(6.9,-1.3);
	\draw [color=gray!60](-6.9,6.5)--($ (-6.9,6.5)!0.5!(6.9,-1.3)$);
	\draw [color=gray!60](6.9,-1.3)--($ (6.9,-1.3)!0.33!(-6.9,6.5)$);
	\draw [color=gray!60](-6.9,3.9)--(6.9,-3.9);
	\draw [color=gray!60](-6.9,3.9)--(-0.8,-6.5);
	\draw [color=gray!60](-6.9,1.3)--(6.9,-6.5);
	\draw [color=gray!60](-6.9,1.3)--(-2.3,-6.5);
	\draw [color=gray!60](-6.9,-1.3)--(2.3,-6.5);
	\draw [color=gray!60](-6.9,-1.3)--(-3.8,-6.5);
	\draw [color=gray!60](-6.9,-3.9)--(-2.3,-6.5);
	\draw [color=gray!60](-6.9,-3.9)--(-5.3,-6.5);
}

\title{Combinatorics of Cone Types in Coxeter groups}

\makeatletter
\markleft{\shorttitle}
\makeatother

\renewcommand{\shorttitle}{\parbox[t]{\textwidth}{\centering Combinatorics of cone types}}

\begin{document}

\author{Yeeka Yau}
\address{School of Mathematics and Statistics, The University of Sydney}
\email{yeeka.yau@sydney.edu.au}

\begin{abstract}
In this article, we establish some new combinatorial properties of cone types in Coxeter groups. Firstly, we show that for any element $x$ in a Coxeter group $W$ and root $\beta$ in its inversion set $\Phi(x)$, the set of elements $y \in W$ satisfying $\Phi(x) \cap \Phi(y) = \{ \beta \} $ is convex in the weak order and admits a unique minimal representative. This is strongly connected to determining the cone type of elements of $W$ and leads to efficient computational methods to determine whether arbitrary elements of $W$ have the same cone type. 

\end{abstract}

\maketitle


\section{Introduction} \label{intro}

Let $(G, S)$ be a group with generating set $S$ and $x \in G$. The \textit{cone type} of $x$ is the set $\{ y \in G \mid \ell(x \cdot y) = \ell(x) + \ell(y) \}$ where $\ell(x)$ is the length of a geodesic from $e$ to $x$ in the Cayley graph of $G$. By the Myhill-Nerode theorem (\cite[Theorem 1.2.9]{epstein1992word}), a group is automatic (i.e. there exists a finite state automaton recognising at least one representative word for each element of $G$) if and only if the set of cone types of $G$ is finite. Thus the set of cone types of a group are fundamental to the automaticity and biautomaticity of the group; both of which are canonical questions in geometric group theory.

In Coxeter groups, cone types exhibit fascinating structure and have recently been extensively studied (see \cite{PARKINSON2022108146}, \cite{ALCO_2024__7_6_1879_0} for example), including the recent positive resolution of the long-standing conjecture that all Coxeter groups are biautomatic \cite{cite-key}. 

In this article, we establish some new combinatorial properties of cone types in Coxeter groups. These new properties lead to efficient computation methods to determine whether arbitrary elements of $W$ have the same cone type. Our first main result is the following fundamental property (see \Cref{thm:witnesses_are_gated}).

\begin{theorem}\label{thm2}
Let $x \in W$ and $\beta \in \Phi(x)$.  
If there exists $y \in W$ such that
\[
\Phi(x) \cap \Phi(y) = \{\beta\},
\]
then there exists a unique element $y_{\beta} \in W$ of minimal length with this property.
\end{theorem}

It turns out that \Cref{thm2} has some important consequences for understanding the structural properties of cone types in Coxeter groups. Let us first fix some terminology and notation before explaining this view.

Let $(W,S)$ be a finitely generated Coxeter system and let $\Phi$ be an associated root system. For $w \in W$, let $\ell(w)$ be the \textit{length} of $w$ and let $\Phi(w)$ be the inversion set of $w$. The \textit{right weak order} on $W$ is defined by $x \preceq y$ if and only if $\ell(y) = \ell(x) + \ell(x^{-1}y)$. Equivalently, in terms of reduced words, $x \preceq y$ if and only if there is a reduced word for $y$ which begins with a reduced word for $x$ (we also call $x$ a \textit{prefix} of $y$). The \textit{right descent set} of $w$ is $D_R(w) = \{ s \in S \mid ws \preceq w \}$ and $\Phi^R(w) = \{ -w\alpha_s \mid s \in D_R(w) \}$ is the corresponding set of \textit{right-descent roots}. 
\par
For $w \in W$ we denote $$T(w) = \{v \in W \mid \ell(wv) = \ell(w) + \ell(v) \}$$ to be the \textit{cone type} of $w$. Intuitively, the cone type of $w$ consists of all elements $v \in W$ such that the concatenation of reduced words for $w$ and $v$ is a reduced word for $w \cdot v$. The set of cone types of $W$ is $\mathbb{T} = \{ T(w) \mid w \in W\}$. A consequence of Brink and Howlett's seminal work in \cite{BH93} is that the set $\mathbb{T}$ is finite and by the Myhill-Nerode theorem, the number of cone types of $W$ is the size (number of states) of the minimal automaton recognising the language of reduced words of $(W,S)$. Fundamentally, it helps to visualise a cone type in terms of the inverses of elements in the following sense:

\begin{example} \label{example:conetype_example}

Consider the Coxeter group of type $\widetilde{G}_2$. The identity is the alcove shaded in dark grey. Let $x$ be the element as indicated below. The cone type of $x^{-1}$, $T(x^{-1})$ is the region shaded in grey below; and can be visualised as the intersection of the half-spaces (the side containing the identity) associated to the walls separating $x$ from the identity $e$.

    	\begin{figure}[H] 
		\centering
		\begin{tikzpicture} [scale=0.4]

			\draw [color=black, line width=1.4pt](0,6.5)--(0,-6.5);
			\draw [color=black, line width=1.4pt](-6.9,1.3)--(6.9,1.3);
			\draw [color=black, line width=2.3pt](-3.8,6.5)--(3.8,-6.5);
			\fill[fill=gray!30] 
			(0,1.3)--(6.9,1.3)--(6.9,-6.5)--(3.8,-6.5)--(0,0);
						
			
			\node[blue] at (-1.3,1.55) {$x$};
			
			\baseforg
		\end{tikzpicture}
		\caption{Example of a cone type in the Coxeter group of type $\widetilde{G}_2$.}
		\label{fig:conetype_example}
	\end{figure}
\end{example}

Note that the cone type in \Cref{fig:conetype_example} is bounded by three walls (or hyperplanes) and that one of the walls separating $x$ from $e$ is redundant for determining $T(x^{-1})$. The roots corresponding to those \textit{boundary} walls are the most precise set of roots which determine the cone type \cite[Theorem 2.8]{PARKINSON2022108146}. These \textit{boundary roots} of $T:=T(x^{-1})$ are characterised in the following way

$$\partial T = \{ \beta \in \Phi^+ \mid \exists \hspace{0.1cm} y \in W \mbox{ with } \Phi(x) \cap \Phi(y) = \{ \beta \} \}$$ 

Let us introduce a few more definitions to illuminate the structure of cone types and illustrate the significance of \Cref{thm2}.

\begin{definition} \label{defn:witnesses_of_beta_wrt_Tx_inverse}
Let $x \in W$ and $\beta \in \Phi^+$. We say an element $y \in W$ is a \textit{witness of $\beta$ with respect to $x$} if 
$$\Phi(x) \cap \Phi(y) = \{ \beta \}$$ 
We denote $\partial T(x^{-1})_{\beta}$ to be the set of \textit{witnesses of $\beta$ with respect to $x$}. More precisely,
\begin{equation*}
    \partial T(x^{-1})_{\beta} = \{ y \in W \mid \Phi(x) \cap \Phi(y) = \{ \beta \} \}
\end{equation*}    
\end{definition}

For a cone type $T$, the \textit{cone type part} $Q(T)$ are the elements $x \in W$ with $T(x^{-1}) = T$.

\begin{remark}
	For each cone type $T$ and $x, y \in Q(T)$ it is not generally true that $ \partial T(x^{-1})_{\beta} =  \partial T(y^{-1})_{\beta}$. Continuing the example from \Cref{fig:conetype_example} consider the following
\end{remark}

	\begin{figure}[H] 
		\centering
		\begin{tikzpicture} [scale=0.44]

			\draw [color=black, line width=1.4pt](0,6.5)--(0,-6.5);
			\draw [color=black, line width=1.4pt](-6.9,1.3)--(6.9,1.3);
			\draw [color=blue, line width=2.3pt](-3.8,6.5)--(3.8,-6.5);
			\fill[fill=gray!30] 
			(0,1.3)--(6.9,1.3)--(6.9,-6.5)--(3.8,-6.5)--(0,0);
			
			\draw [color=blue!30, line width=3.2pt](0.13,-0.2)--(0.13,-6.5);
			\draw [color=blue!30, line width=4.2pt](0.05,-0.2)--(3.7,-6.5);
			
			\draw [color=red!30, line width=2.2pt](0.2,-2.7)--(2.4,-6.5);
			\draw [color=red!30, line width=2.2pt](0.2,-0.6)--(3.6,-6.5);
			\draw [color=red!30, line width=2.2pt](0.2,-0.6)--(0.2,-2.7);
						
			\draw[dashed,line width=1.6pt] [color=black](-6.9,6.5)--($ (-6.9,6.5)!0.34!(6.9,-1.3)$);
			\draw[dashed,line width=1.6pt] [color=black](-2.3, 3.9)--($(-2.3, 3.9)!0.635!(0,0)$);
			\draw[dashed,line width=1.6pt] [color=black](-6.9, 3.9)--($(-6.9, 3.9)!0.97!(-2.3,1.3)$);
			\draw[dashed,line width=1.6pt] [color=black](-2.3,1.3)--($(-2.3,1.3)!1.2!(-1,1.3)$);
			
			\node[blue] at (-1.3,1.55) {$x$};
            \node[red] at (-2.97,2.05) {$y$};
			
			\baseforg
		\end{tikzpicture}
		\caption{Let $x$ and $y$ be the elements illustrated above. One can see that $T(x^{-1}) = T(y^{-1})$, so $x$ and $y$ are in the same cone type part $Q(T)$; which is the region bounded by the dashed lines, and the horizontal black and blue hyperplanes (the region looks like a downward right pointing arrow). The cone type $T:= T(x^{-1}) = T(y^{-1})$ is represented by the gray shaded region. Let $\beta \in \partial T$ correspond to the blue hyperplane. The region bounded by the red highlighted lines is $\partial T(y^{-1})_{\beta}$ and the region bounded by the blue highlighted lines is $ \partial T(x^{-1})_{\beta}$.}
		\label{fig:witnesses_are_different}
	\end{figure}

Since for elements $x, y$ in the same cone type part $Q(T)$ we do not necessarily have $ \partial T(x^{-1})_{\beta} = \partial T(y^{-1})_{\beta}$ it is natural to make the following definition at the level of the cone type.

\begin{definition}
	Let $T \in \mathbb{T}$ be a cone type and $\beta \in \partial T$. The \textit{witnesses of $\beta$ with respect to $T$} is 
	\begin{align*}
		\partial T_{\beta} = &\ \{ y \in W \mid \Phi(x) \cap \Phi(y) = \{ \beta \} \mbox{ for all } x \in Q(T) \} \\
        = &\ \bigcap_{x \in Q(T)} \partial T(x^{-1})_{\beta}
	\end{align*}  
\end{definition}

In \cite{PARKINSON2022108146} we showed that for each boundary root $\beta \in \partial T$ and $x \in Q(T)$ there is \textit{a} witness $y \in W$ such that $\Phi^R(y) = \{ \beta \}$ (see \Cref{minimal_witness}). However, it was not known whether such an element $y$ is unique, let alone that it is the minimal length representative of $\partial T_{\beta}$.

As a consequence of \Cref{thm2} we obtain the following.

\begin{corollary} \label{cor:witnesses_of_boundary_roots}
	For each cone type $T$ and $\beta \in \partial T$ there is a unique minimal length element $y_{\beta} \in W$ such that $$\Phi(x) \cap \Phi(y_{\beta}) = \{ \beta \}$$ 
	for all $x \in Q(T)$. Moreover, $\Phi^R(y) = \{ \beta \}$ and $y_{\beta} \preceq y$ for all $y \in \partial T(x^{-1})_{\beta}$ for each $x \in Q(T)$. Thus $y_{\beta} \preceq y$ for all $y \in \partial T_{\beta}$.
\end{corollary}

We give a special name for the property enjoyed by the element $y_{\beta}$ in relation to the sets $\partial T(x^{-1})_{\beta}$ and $\partial T_{\beta}$. Following the terminology introduced in \cite{PARKINSON2022108146}, for a subset $X \subseteq W$, we say $X$ is \textit{gated} if there is a unique minimal length element $x \in X$ such that $x \preceq y$ for all $y \in X$. We say the subset $X$ is \textit{convex} if for all  $x,y \in X$ and all reduced expressions $x^{-1}y = s_1 \cdots s_n$, the elements $xs_1 \cdots s_j$ with $0 \le j \le n$ is in $X$. In terms of the Cayley graph $X^1$ of $(W,S)$ this means that all elements on geodesics between $x$ and $y$ are in $X$.

\Cref{thm2} and \Cref{cor:witnesses_of_boundary_roots} reveals that each cone type $T$ is associated with a number of convex, gated subsets of $W$, including its cone type part $Q(T)$ (see \Cref{thm:gates_of_conetype_partition}), the sets $\partial T(x^{-1})_{\beta}$ for $x \in Q(T)$ and $\beta \in \partial T$ (as well as the cone type $T$ itself with the identity as the gate). 

We illustrate in \Cref{fig:witness_example_g2} an example in the case of the Coxeter group of type $\widetilde{G}_2$. The figure shows a cone type $T$, its cone type part $Q(T)$ and the sets $\partial T(x^{-1})_{\beta}$ for each $\beta \in \partial T$ for the minimal element $x$ of $Q(T)$.

\begin{figure}[H] 
    \centering
        \begin{tikzpicture} [scale=0.5]
            \fill[fill=gray!30] 
            (-6.9,6.5)--(0,2.6)--(0.8,1.3)--(0,0)--(-6.9,-3.9);
            \fill[fill=blue!20] 
            (-6.9,6.5)--(0,2.6)--(-2.3,6.5);
            \fill[fill=red!20] 
            (0,0)--(-6.9,-3.9)--(-6.9,-6.5)--(-3.8,-6.5);
            \fill[fill=green!20] 
            (0.8,1.3)--(0,0)--(1.1,0.6);
            \fill[fill=cyan!20] 
            (0.8,1.3)--(0,2.6)--(1.13,2);
            \fill[fill=blue] (-0.65,3.2) circle (4pt);
            \fill[fill=cyan] (0.7,1.85) circle (4pt);
            \definecolor{darkgreen}{rgb}{0.0, 0.5, 0.0}
            \fill[fill=darkgreen] (0.7,0.7) circle (4pt);
            \fill[fill=red] (-0.65,-0.65) circle (4pt);
            \draw[line width=1.5pt] [color=black](-2.3,6.5)--(5.3,-6.5);
            \draw[line width=1.5pt] [color=black](-6.9,6.5)--(6.9,-1.3);
            \draw[line width=1.5pt] [color=black](6.9,3.9)--(-6.9,-3.9);
            \draw[line width=1.5pt] [color=black](3.8,6.5)--(-3.8,-6.5);
            \draw[dashed,line width=1.2pt] [color=black](2.5,1.3)--(6.9,1.3);
            \draw[dashed,line width=1.2pt] [color=black](4.5,0)--(6.9,0);
            \fill[fill=yellow!20] 
            (2.5,1.3)--(6.9,1.3)--(6.9,0)--(4.5,0);
            \node[black] at (3.3,1) {$x$};
            \baseforg
        \end{tikzpicture}
    \caption{Let the element $x$ be indicated as above. The grey shaded region is the cone type $T:=T(x^{-1})$ (the darker grey alcove represents the identity) with boundary roots $\partial T$ indicated by the thickened black walls. The yellow shaded region is the cone type part $Q(T)$. The remaining coloured shaded regions are the sets $\partial T(x^{-1})_{\beta}$ for each $\beta \in \partial T$ and the corresponding coloured dots are the gates of those regions.}
    \label{fig:witness_example_g2}
\end{figure}

Let us note a few more properties of the unique minimal length witnesses of boundary roots. One may notice that by symmetry, $y \in \partial T(x^{-1})_{\beta}$ if and only if $x \in \partial T(y^{-1})_{\beta}$. Furthermore, the gates of the sets $Q(T)$ for $T \in \mathbb{T}$ (we denote these elements by $\Gamma$) are precisely the elements $y \in W$ such that $\Phi^R(y) \subseteq \partial T(y^{-1})$ (see \Cref{thm:gates_characterised_by_phi0}). It thus follows by \Cref{cor:witnesses_of_boundary_roots} that the unique minimal witnesses of boundary roots are gates of cone type parts. In particular, since these minimal witnesses have a single right-descent, they are what we call \textit{tight gates} denoted $\Gamma^0$. More precisely

\begin{equation*}
    \Gamma^0 = \{x \in \Gamma \mid |\Phi^R(x)| = 1 \}
\end{equation*}

\begin{theorem} \label{thm:tight_gates_are_join_irreducible}
	Let $\Gamma$ be the gates of the cone type parts $\{ Q(T) \mid T \in \mathbb{T} \}$. 
    Then
    $$\Gamma^0 = \{ g \in \Gamma \mid |\Phi^R(g)| = 1\}$$ is the set of  join-irreducible elements of the partially ordered set $(\Gamma, \preceq)$.
\end{theorem}

A consequence of \Cref{thm:tight_gates_are_join_irreducible} is that the set of tight gates $\Gamma^0$ completely determines the set of cone types $\mathbb{T}$ of $W$ (see \Cref{cor:every_conetype_is_expressible_as_an_intersection_of_tightgates}). In addition, we obtain the following (see \Cref{suffix_closure_of_tight_gates}).

\begin{corollary} \label{cor1}
The set
    $$\Gamma^0 = \{ g \in G \mid |\Phi^R(g)| = 1\}$$
    is closed under suffix.
\end{corollary}

\Cref{cor1} then leads to a very efficient algorithm to compute the set $\Gamma^0$ without having to compute the entire set $\Gamma$ (effectively all the cone types) first. A consequence of this result is that to determine whether two elements $x, y$ have the same cone type, one is only required to compute $\Gamma^0$, a set strictly smaller than $\Gamma$ (with only a few exceptions). Our computations using Sagemath (\cite{sagemath}) show that this set is in many cases much smaller than $\Gamma$ as the rank of $W$ grows (see \Cref{algo:compute_tight_gates_and_super_elementary_roots} and \Cref{table:data_for_tight_gates}).

\section{Preliminaries} \label{section:Preliminaries}

To begin, we review some essential properties of Coxeter systems, the standard geometric representation, the Coxeter complex, Cayley graph and cone types. For general Coxeter group theory the primary references are \cite{Hu90}, \cite{BB05} and \cite{Ronan_2009}. For more on cone types see \cite{PARKINSON2022108146}.

\subsection{Coxeter Groups} \label{Coxeter systems and root systems}

Let $(W,S)$ be a Coxeter system with $|S| < \infty$. 
The \textit{length} of $w\in W$ is 
$$
\ell(w)=\min\{n\geq 0\mid w=s_1 \ldots s_n\text{ with }s_1,\ldots,s_n\in S\},
$$
and an expression $w=s_1 \ldots s_n$ with $n=\ell(w)$ is called a \textit{reduced expression} or \textit{reduced word} for~$w$. The \textit{left descent set} of $w$ is $D_L(w) = \{ s \in S \mid \ell(sw) = \ell(w) - 1 \}$, similarly the \textit{right descent set} is $D_R(w) = \{ s \in S \mid \ell(ws) = \ell(w) - 1 \}$. It is well known that the set
\begin{equation*}
    R = \{ wsw^{-1} \mid w \in W, s\in S \}
\end{equation*}
is the set of \textit{reflections} of $W$ (see \cite[Chapter 1]{BB05}).

\subsubsection{Standard Geometric Representation and root system of $W$} \label{root_system}

Let $V$ be an $\mathbb{R}$-vector space with basis $\{\alpha_s\mid s\in S\}$. Define a symmetric bilinear form on $V$ by linearly extending $\langle\alpha_s,\alpha_t\rangle=-\cos(\pi/m(s,t))$. The Coxeter group $W$ acts on $V$ by the rule $sv =v-2\langle v,\alpha_s\rangle \alpha_s$ for $s\in S$ and $v\in V$, and the root system of $W$ is $\Phi=\{w\alpha_s \mid w\in W,\,s\in S\}$. The elements of $\Phi$ are called \textit{roots}, and the \textit{simple roots} are $\Delta = \{ \alpha_s \mid s \in S \}$.
\par
Each root $\alpha\in\Phi$ can be written as $\alpha=\sum_{s\in S}c_s\alpha_s$ with either $c_s\geq 0$ for all $s\in S$, or $c_s\leq 0$ for all $s\in S$. In the first case $\alpha$ is called \textit{positive} (written $\alpha>0$), and in the second case $\alpha$ is called \textit{negative} (written $\alpha<0$). Denote $\Phi^+$ (resp. $\Phi^-$) to be the positive roots (resp. negative roots) of $\Phi$. By definition of $\Phi$, each root $\beta$ can be written $\beta = w\alpha_s$ for some $s \in S$. The reflection $s_{\beta} = wsw^{-1}$ gives rise to the hyperplane in $V$
\begin{equation*}
    H_{\beta} = \{ \lambda \in V \mid s_{\beta}\lambda = \lambda\}
\end{equation*}
and this hyperplane separates $W$ into two sets (called \textit{half spaces})
\begin{equation*}
    H_{\beta}^{-} = \{w \in W \mid \ell(s_{\beta}w) < \ell(w) \} \mbox{ and } H_{\beta}^{+} = \{w \in W \mid \ell(s_{\beta}w) > \ell(w) \}
\end{equation*}
Half spaces provide an important characterisation of  convexity in $W$.

\begin{lemma} \cite[Proposition 3.94]{buildings} \label{lem:convexity}
    A subset $A \subseteq W$ is convex if and only if it is an intersection of half spaces.
\end{lemma}
For $A \subset \Phi^+$, $cone(A)$ is the set of non-negative linear combinations of roots in $A$ and $cone_{\Phi}(A) = cone(A) \cap \Phi^+$. The \textit{(left) inversion set} of $w\in W$ is 
$$
\Phi(w)=\{\alpha\in \Phi^+\mid w^{-1}(\alpha)<0\}.
$$

Geometrically, the inversion set of $w$ corresponds to the roots $\beta$ for which $w$ is contained in the half space $H_{\beta}^-$ (or equivalently, separated from the identity by the hyperplane $H_{\beta}$). A well known fact is that $x \preceq y$ if and only if $\Phi(x) \subseteq \Phi(y)$ (\cite[Proposition 2.8(3)]{HD16}). An important subset of $\Phi(w)$ is the set of \textit{short inversions} $\Phi^1(w)$ (also called the \textit{base} of $\Phi(w)$).

\begin{definition}\cite[Proposition 4.6]{HD16} 
Let $w \in W$. The \textit{short inversions} of $\Phi(w)$ is the set
\begin{equation*}
    \Phi^1(w) = \{ \beta \in \Phi(w) \mid \ell(s_{\beta}w) = \ell(w) - 1 \}
\end{equation*}
\end{definition}

The set $\Phi^1(w)$ determines the inversion set $\Phi(w)$ in the following way (see \cite[Lemma 1.7]{dyer_2019} and \cite[Corollary 2.13]{HOHLWEG20161}).

\begin{proposition} \label{inversion_set_is_determined_by_phi1}
For $w \in W$ we have $\Phi(w) = cone_{\Phi}(\Phi^1(w))$ and if $A \subset \Phi^+$ is such that $\Phi(w) = cone_{\Phi}(A)$ then $\Phi^1(w) \subseteq A$.
\end{proposition}

As noted earlier, another important subset of $\Phi(w)$ is the following.

\begin{definition} \label{defn:final_roots}
    Let $w \in W$. The set of \textit{right-descent roots} of $w$ (also called \textit{final roots}) is the set
    \begin{equation*}
        \Phi^R(w) = \{ -w\alpha_s \mid s \in D_R(w) \}
    \end{equation*}
\end{definition}

A root $\beta \in \Phi^R(w)$ if and only if $s_{\beta}w = ws$ for some $s \in D_R(w)$. In terms of the Coxeter complex $\Sigma(W,S)$ or Cayley graph $X^1$ the right descent roots correspond to the set of \textit{final} or \textit{last} walls crossed for geodesics from the identity $e$ to $w$ (see \Cref{the_coxeter_complex} and \Cref{the_cayley_graph}). The \textit{left-descent roots} are $\Phi^L(w) = \{ \alpha_s \mid s \in D_L(w) \}$ and correspond to \textit{initial} walls crossed for geodesics from the identity $e$ to $w$.
\par
We collect some further useful facts regarding joins in the right weak order which are used in this article.

\begin{proposition}\cite[Proposition 2.8]{HD16} \label{inversion_set_of_join}
If $X \subset W$ is bounded (in the right weak order), then
\begin{equation*}
    \Phi(\bigvee X) = cone_{\Phi}(\bigcup_{x \in X} \Phi(x))
\end{equation*}
\end{proposition}

\begin{corollarySec}\cite[Corollary 4.7 (2)]{HD16} \label{cor:phi1_of_join_isphi1_of_x}
    If $x \vee y$ exists and $\ell(s_{\beta}(x \vee y)) = \ell(x \vee y ) - 1$, then either $\ell(s_{\beta}x) = \ell(x) -1$ or $\ell(s_{\beta}y) = \ell(y) -1$.
\end{corollarySec}

\subsubsection{The Coxeter complex} \label{the_coxeter_complex}

The Coxeter complex $\Sigma(W,S)$ is a useful combinatorial object for providing geometric intuition for many of the concepts discussed and illustrated in the figures of this article. While the formal construction of $\Sigma(W,S)$ is not required for any result in this work, it will be helpful to provide a brief description. The reader is directed to \cite{buildings} for further details.

For each $w \in W$, let $C_w$ be a combinatorial simplex with vertices indexed by $S$. For each $w \in W$ and $s \in S$ identify the faces of the simplices $C_w$ and $C_{ws}$ via their matching vertices in $S \setminus \{ s \}$. The resulting simplicial complex $\Sigma(W,S)$ is known as the \textit{Coxeter complex} of $(W,S)$. The Coxeter complex has maximal simplices $C_w$, $w \in W$ called the \textit{alcoves} of the complex which we identify with the elements of $W$. The Coxeter group $W$ then acts on the alcoves simply transitively whereby $wC_v = C_{w \cdot v}$ for all $w,v \in W$.

For $\beta \in \Phi^+$ the set of simplices
\begin{equation*}
    H_{\beta} = \{ \sigma \in \Sigma(W,S) \mid s_{\beta}(\sigma) = \sigma\}
\end{equation*}

is a \textit{wall} of the Coxeter complex. Note that no alcoves are contained in a wall, and thus one may consider the wall $H_{\beta}$ separating the half-spaces $H_{\beta}^+$ and $H_{\beta}^-$ (the notion of half-spaces here coincide with that of half-spaces discussed in \Cref{root_system} and walls play the same role as hyperplanes). In particular, for $w \in W$, its inversion set $\Phi(w)$ bijectively corresponds to the set of walls separating the alcove $C_w$ from $C_e$ and the right descent roots $\Phi^R(w)$ are the walls incident to the alcove $C_w$ which separate $C_w$ from $C_e$.

\subsubsection{The Cayley graph of $(W,S)$} \label{the_cayley_graph}
The Cayley graph of $(W,S)$ is the dual graph of the Coxeter complex, recording adjacency of alcoves across walls labelled by simple reflections. Thus,
geometric features of the Coxeter complex translate naturally into combinatorial properties
of paths and convexity in the Cayley graph.

More precisely, denote $X^1$ to be the Cayley graph of $W$ with vertex set $X^0 = W$ and edge-set $\{ (w, ws) \mid w \in W, s \in S \}$ where each edge is of length 1. For vertices $x,y \in X^1$ a \textit{geodesic} between $x$ and $y$ is a minimal length path between $x$ and $y$ in $X^1$. In the setting of $X^1$, for $w \in W$, $\ell(w)$ is the length of geodesics from the identity $e$ to $w$.

The left action of $W$ on $X^0$ induces an action of $W$ on $X^1$. Under the identification of the Cayley graph with the dual graph of the Coxeter complex,
walls admit a consistent interpretation in both settings.
For a reflection $r \in R$, the \textit{wall} $H_r$ is defined as the fixed point set of $r$ acting
on the Cayley graph $X^1$. Equivalently, an edge $(w,ws)$ crosses $H_r$ if and only if
$r = wsw^{-1}$, that is, if the corresponding adjacent alcoves are separated by the
reflecting wall associated to $r$ in the Coxeter complex (see \Cref{fig:coxeter_complex_with_cayley_graph}). 

Under the standard geometric representation, each root $\beta \in \Phi$ determines a
reflecting hyperplane $H_\beta \subset V$, whose induced wall $H_{s_{\beta}}$ in the Coxeter complex and fixed-point set in the Cayley graph define the same separation of $W$ into half-spaces. Throughout this article, most arguments are phrased in terms of roots and inversion sets; however, in the proofs of \Cref{thm:witnesses_are_gated}  and \Cref{prop:bijection_join_refine_phi1} we make essential use of the Cayley graph description, which may also be visualised geometrically via the Coxeter complex.

\begin{figure}[H]
    \centering
    \begin{tikzpicture}
        \basefora
        \cayleyfora
    \end{tikzpicture}
    
    \caption{Coxeter complex for the Coxeter group of type $\widetilde{A}_2$ with Cayley graph overlaid.}
    \label{fig:coxeter_complex_with_cayley_graph}
\end{figure}

Two distinct walls $H_r$ and $H_q$ \textit{intersect} if $H_r$ is not contained in a half-space for $H_q$ (this relation is symmetric), or equivalently, $\langle r, q \rangle$ is finite. We say $r, q$ are \textit{sharp-angled} if $r$ and $q$ do not commute and there is $w \in W$ such that $wrw^{-1}$ and $wqw^{-1}$ are in $S$. Then there is a component of $X^1 \setminus (H_r \cup H_q)$ whose intersection with $X^0$ is a \textit{geometric fundamental domain} for the action of $\langle r, q \rangle$ on $X^0$.

For $J \subseteq R$ and $w \in X^0$ the $J(w)$-\textit{residue} is the subgraph of $X^1$ induced by the action of the subgroup generated by $J$ on $w$. In particular, note that for $w \in W$ if $\alpha, \beta \in \Phi^R(w)$ then $w^{-1}s_{\alpha}w = s$ and $w^{-1}s_{\beta}w = t$ are in $S$ and $\langle s_{\alpha}, s_{\beta} \rangle$ is a finite dihedral reflection subgroup with $2m(s,t)$ elements. Then $w\Phi_{\langle s, t \rangle}^+$ is the set of $m(s,t)$ roots separating the $2m(s,t)$ elements of the residue $\langle s_{\alpha}, s_{\beta} \rangle(w) = wW_{\langle s, t \rangle}$. We will use this rank--two residue structure in path--modification arguments in the proofs of \Cref{thm:witnesses_are_gated} and \Cref{prop:bijection_join_refine_phi1}.

In those proofs, we also consider an order on paths in $X^1$ using the lexicographic order on tuples with entries in $\mathbb{N}$. For tuples $\textbf{a} = (a_1, \ldots, a_m)$ and $\textbf{b} = (b_1, \ldots, b_n)$ we have $\textbf{a} \le \textbf{b}$ if and only if one of the following holds:
\begin{enumerate}
    \item[(i)] There exists $j \le \min(m, n)$ such that $a_i = b_i$ for all $i < j$ and $a_j < b_j$, or;
    \item[(ii)] $a_i = b_i$ for all $i \le m$ and $m < n$.
\end{enumerate}
For a path $\pi = (p_0, \ldots, p_n)$ in $X^1$ define the ordered tuple $\Pi = (n_L, \ldots, n_1)$ where $n_j$ is the number of elements $p_i \in \pi$ with $\ell(p_i) = j$ and $L = max_{i=0}^{n} \ell(p_i)$. Then for paths $\pi, \pi'$ we have $\pi \le \pi'$ if and only if $\Pi \le \Pi'.$

\subsubsection{Cone Types and Boundary roots of $W$}
 The definition of cone types, cone type parts and boundary roots were introduced in \Cref{intro}. We now recall additional structure and properties of cone types in Coxeter groups relevant to the results in this article.

The \textit{cone type arrangement} or \textit{cone type partition} $\mathscr{T}$ is the partition of $W$ into its cone type parts.
\begin{equation*}
    Q(T) = \{x \in W \mid T(x^{-1}) = T \}
\end{equation*}
for each $T \in \mathbb{T}$ (see \Cref{fig:g2_conetype_partition} for an example). Each part of this partition can be considered as a \textit{state} of the minimal finite state automaton recognising the language of reduced words of $(W,S)$ (this fact isn't necessarily required, but is discussed in some depth in \cite{PARKINSON2022108146} for interested readers). The following theorem is an interesting characterisation of $\mathscr{T}$ and reveals that the parts $\{ Q(T) \mid T \in \mathbb{T} \}$ of $\mathscr{T}$ can realised as intersections of the set of cone types $\mathbb{T}$.

\begin{theoremSec} \label{thm:conetypepartition_asintersectionofconetypes}
    \cite[Theorem 3.4.3]{Yau21}
    For $x,y \in W$ define $x \sim y$ if and only if, for every cone type  $T \in \mathbb{T}$, one has 
    \begin{equation*}
        x \in T \mbox{ if and only if } y \in T
    \end{equation*}
    Then the partition of $W$ by the $\sim$ equivalence classes is $\mathscr{T}$.
\end{theoremSec}

As mentioned in \Cref{intro}, each part $Q(T)$ contains a unique minimal length representative $g_T$, called the \textit{gate} of $Q(T)$ which is a prefix of all elements in $Q(T)$ (these are the grey shaded alcoves in the example of  \Cref{fig:g2_conetype_partition}), this was proved in \cite[Theorem 1]{PARKINSON2022108146} and \cite[Theorem 1.3]{ALCO_2024__7_6_1879_0}. These elements are also known in the literature as the \textit{gates of $\mathscr{T}$} and are denoted by $\Gamma$. Moreover, we showed in \cite[Corollary 1.5]{ALCO_2024__7_6_1879_0} that $\Gamma$ is the smallest \textit{Garside shadow} in $W$; the smallest subset of $W$ containing $S$ and is closed under join and suffix in the right weak order. 

The notion of Garside shadows will only be briefly discussed again in \Cref{sec:computing_gamma0}; we encourage the reader to see \cite{HOHLWEG2016431}, \cite{HD16} and \cite{santos2025garsideshadowsbiautomaticstructures} for more on these fascinating structures.

 \begin{theoremSec}\cite[Theorem 1]{PARKINSON2022108146} \label{thm:gates_of_conetype_partition}
     For each cone type $T$ there is a unique element $m_T \in W$ of minimal length such that $T(m_T) = T$. Moreover, if $w \in W$ with $T(w) = T$ then $m_T$ is a suffix of $w$.
 \end{theoremSec}

\begin{remark}
    Note that in \Cref{thm:gates_of_conetype_partition}, $m_T = g_T^{-1}$ where $g_T$ is the gate of the part $Q(T)$.
\end{remark}

\begin{figure}
    	\includegraphics[scale=0.18]{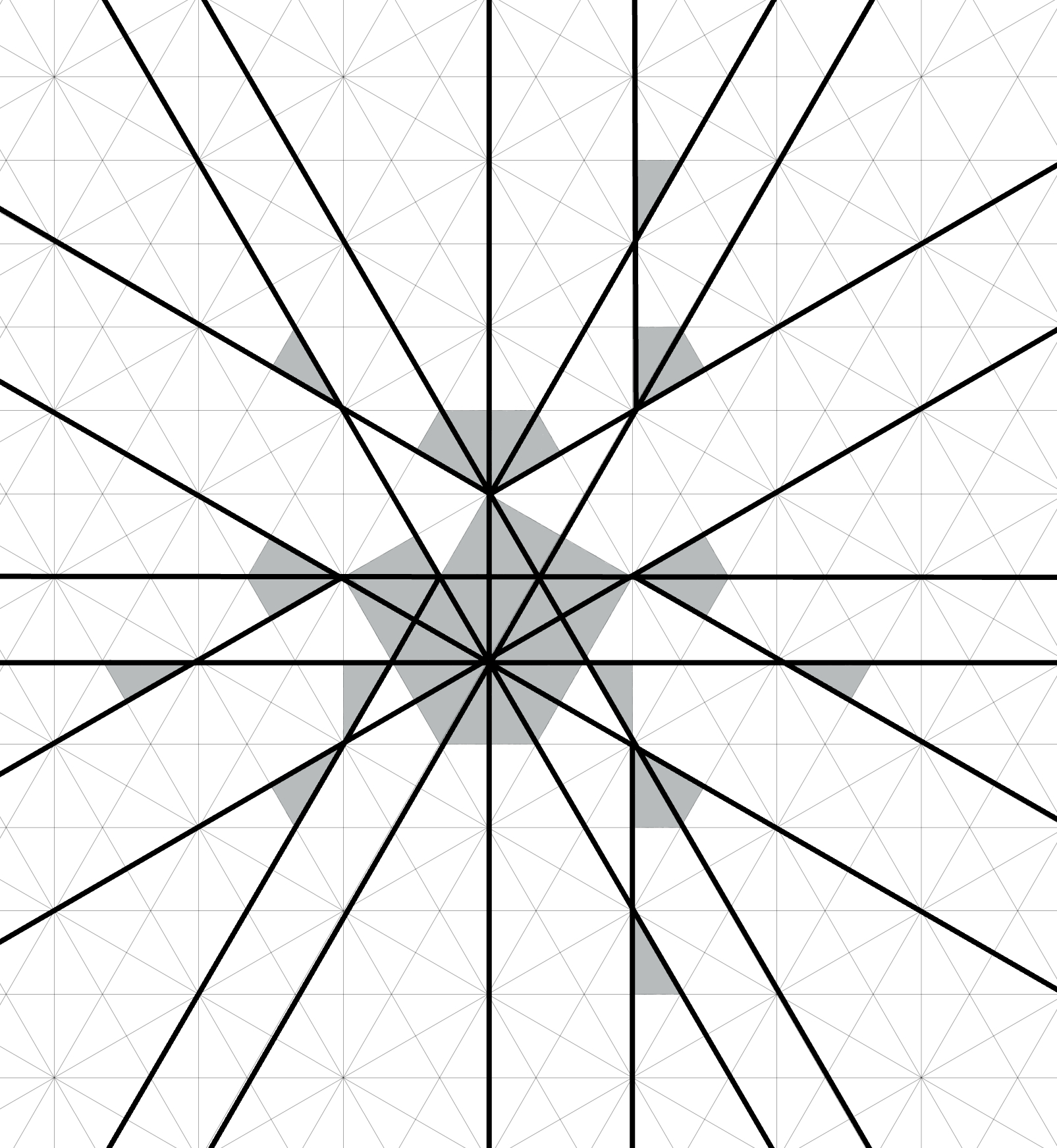}

    \caption{The cone type partition $\mathscr{T}$ for the Coxeter group of type $\widetilde{G}_2$. Each part $Q(T)$ contains the elements $x$ such that $T(x^{-1}) = T$. The shaded in alcove of each part is the gate of its part $Q(T)$.}
    \label{fig:g2_conetype_partition}
\end{figure}

The \textit{tight gates} of $\mathscr{T}$ is the set
$$\Gamma^0 = \{ g \in \Gamma \mid |\Phi^R(g)| = 1\}$$

The following theorem gives a characterisation of a gate $x$ in terms of $\Phi^R(x)$.

\begin{theoremSec}\cite[Theorem 4.33]{PARKINSON2022108146} \label{thm:gates_characterised_by_phi0}
    Let $x \in W$. Then $x \in \Gamma$ if and only if for each $\beta \in \Phi^R(x)$ there exists $w \in W$ with $\Phi(x) \cap \Phi(w) = \{ \beta \}$.
\end{theoremSec}

We list some further elementary results regarding cone types and their boundary roots which will be used throughout. We omit their proofs here.

\begin{lemma}\cite[Lemma 1.15]{PARKINSON2022108146} \label{lem:suffix_conetype_containment}
	If $ x \preceq y$ then $T(y^{-1}) \subseteq T(x^{-1})$.
\end{lemma}

The following result formally states the geometric interpretation of cone types as intersection of half-spaces (as illustrated in \Cref{example:conetype_example}) and the next relates the join of a set of elements with their cone types.

\begin{theoremSec} \cite[Corollary 2.9]{PARKINSON2022108146} \label{thm:cone_type_formulas}
    Let $T$ be a cone type. If $T= T(x^{-1})$ Then 
    \begin{equation*}
        T = \bigcap_{\Phi(x)} H_{\alpha}^{+} =
        \bigcap_{\Phi^1(x)} H_{\alpha}^{+} = \bigcap_{\partial T(x^{-1})} H_{\alpha}^{+}
    \end{equation*}
\end{theoremSec}

 \begin{proposition}\cite[Proposition 1.19]{PARKINSON2022108146} \label{cone_type_of_join}
     If $X \subseteq W$ is bounded with $y = \bigvee X$ then $T(y^{-1}) = \bigcap_{x \in X} T(x^{-1})$.
 \end{proposition}

We end this section by recalling two results describing boundary roots of cone types and the existence of minimal witnesses. These results provide the foundational input for \Cref{thm2} and \Cref{cor:witnesses_of_boundary_roots}, where we strengthen
existence statements to uniqueness, canonical minimality and \textit{gatedness}.

\begin{theoremSec} \cite[Theorem 2.6]{PARKINSON2022108146} \label{thm:boundary_roots_definition}
   Let $T$ be a cone type. If $T = T(x^{-1})$ then $\beta \in \partial T$ if and only if there exists $w \in W$ with 
   \begin{equation*}
       \Phi(x) \cap \Phi(w) = \{ \beta \}
   \end{equation*}
   Moreover, if $\beta \in \partial T$ then there exists $w \in W$, independent of $x$, such that $\Phi(x) \cap \Phi(w) = \{ \beta \}$ whenever $T = T(x^{-1})$.
\end{theoremSec}

\begin{proposition} \cite[Proposition 4.34]{PARKINSON2022108146} \label{minimal_witness}
    Let $x \in W$ and $\beta \in \Phi^+$. Suppose there exists $w \in W$ such that $\Phi(x) \cap \Phi(w) = \{\beta\}$, and let $w$ be of minimal length subject to this property. Then $\Phi^R(w) = \{ \beta \}$ and $w$ is a gate.
\end{proposition}

\section{Witnesses of boundary roots} \label{sec: Witnesses of boundary roots}

In this section, we show that for any element $x \in W$ and $\beta \in \Phi(x)$, the set of elements $y \in W$ satisfying $\Phi(x) \cap \Phi(y) = \{ \beta \} $ is convex in the right weak order and admits a unique minimal representative.  By \Cref{defn:witnesses_of_beta_wrt_Tx_inverse} these are the  \textit{witnesses} of $\beta$ with respect to $x$ and means that $\beta$ is a \textit{boundary root} of the cone type $T := T(x^{-1})$. 

Furthermore, we show that this minimal witness is the same one for all $z \in Q(T)$. In addition, we show that the set of \textit{tight gates} $\Gamma^0$ is closed under suffix. We begin with a basic characterisation of the set of witnesses for a boundary root in terms of intersections of half spaces.

\begin{lemma} \label{witnesses_are_convex}
    Let $x \in W$ and $T:= T(x^{-1})$. Then
    \begin{equation*}
        \partial T(x^{-1})_{\beta} = H_{\beta}^{-} \cap \big( \bigcap_{\alpha \in \Phi(x) \setminus \{ \beta \} } H_{\alpha}^{+} \big)
    \end{equation*}
    and the set $\partial T(x^{-1})_{\beta}$ is convex.
\end{lemma}
\begin{proof}
    The proof is straightforward. If $w \in \partial T(x^{-1})_{\beta}$ then $\alpha \notin \Phi(w)$ for all $\alpha \neq \beta \in \Phi(x)$, so $w \in H_{\alpha}^+$ for all $\alpha \in \Phi(x) \setminus \{ \beta \}$. Since $\beta \in \Phi(w)$ we have $w \in H_{\beta}^-$. The reverse inclusion is also clear. Convexity follows by \Cref{lem:convexity} since $\partial T(x^{-1})_{\beta}$ is an intersection of half spaces.
\end{proof}

\begin{theoremSec} \label{thm:witnesses_are_gated}
Let $x \in W$ and $\beta \in \Phi(x)$ such that there exists $y \in W$ with $\Phi(x) \cap \Phi(y) = \{\beta\}$. Then there exists a unique minimal length element $y_{\beta} \in W$ such that 
\begin{equation*}
	\Phi(x) \cap \Phi(y_{\beta}) = \{ \beta \}
\end{equation*}
Furthermore, $y_{\beta} \in \Gamma^0$ and the set of witnesses $\partial T(x^{-1})_{\beta}$ is a convex, gated set in $W$.
\end{theoremSec}
\begin{proof}
    This proof is inspired by the proof of  \cite[Theorem 1.1]{ALCO_2024__7_6_1879_0}. To set things up, let us pass between reflections in the Cayley graph $X^1$ of $W$ and their corresponding roots in $\Phi^+$. Hence consider $\partial T(x^{-1})_{\beta}$ to be the intersection of the half spaces $H_{t}^-$ ($t := s_{\beta}$) and $H_{u_{i}}^+$ ($u_i := s_{\alpha_{i}}$) in $X^1$ corresponding to the roots $\beta$ and $\{ \alpha_i  \mid \alpha_i \in \Phi(x) \setminus \{ \beta \} \}$ (per the characterisation of $\partial T(x^{-1})_{\beta}$ in \Cref{witnesses_are_convex}).
    \par
    It suffices to show that for each $y_0, y_n \in \partial T(x^{-1})_{\beta}$ there is $y \in \partial T(x^{-1})_{\beta}$ with $y_0 \succeq y \preceq y_n$. Let $\pi = (y_0, y_1 , \ldots , y_n)$ be a geodesic edge path between $y_0$ and $y_n$. By \Cref{witnesses_are_convex}, each $y_i \in \partial T(x^{-1})_{\beta}$ for $0 \le i \le n$.
    \par
    Our aim is to modify $\pi$ to obtain another embedded edge path between $y_0$ and $y_n$ (not necessarily geodesic) so that there is no $y_{i-1} \prec y_i \succ y_{i+1}$. Now suppose there is $y_{i-1} \prec y_i \succ y_{i+1}$. Let $H_r, H_q$ be the walls separating $y_{i-1}$ from $y_i$ and $y_i$ from $y_{i+1}$ respectively and let $\alpha_r, \alpha_q$ be the corresponding roots. Note that $\alpha_r \neq \beta$ and $\alpha_q \neq \beta$ since by convexity $y_{i-1}, y_i, y_{i+1} \in \partial T(x^{-1})_{\beta}$. Also, $y_{i-1}, y_{i+1}$ are elements of the \textit{residue} $R = \langle r, q \rangle (y_i)$ and $\alpha_r, \alpha_q \in \Phi^R(y_i)$. If $r$ and $q$ do not commute then $r,q$ are sharp angled and the identity $e$ is in a geometric fundamental domain for the finite dihedral reflection subgroup $W_{\langle r, q \rangle}$ generated by the reflections $r,q$. We claim that all the elements of the residue $R = \langle r, q \rangle (y_i)$ lie in $\partial T(x^{-1})_{\beta}$. 
    \par
    Since $y_i \in \partial T(x^{-1})_{\beta}$ this implies that $\alpha_r, \alpha_q \notin \Phi(x) \setminus \{ \beta \}$. We then need to show that $\alpha' \notin \Phi(x) \setminus \{ \beta \}$ for any $\alpha' \in \Phi^{+}_{\langle r, q \rangle}$. Hence we can assume that $r$ and $q$ do not commute, otherwise there is no other root in $\Phi^{+}_{\langle r, q \rangle}$. If there is $\alpha' \in \Phi^{+}_{\langle r, q \rangle} \cap \big( \Phi(x) \setminus \{ \beta \} \big)$ then since $\alpha' = k \alpha_r + j \alpha_q$ for some $k,j \in \mathbb{R}_{> 0}$ this implies that either $\alpha_r \in \Phi(x) \cap \Phi(y_i)$ or $\alpha_q \in \Phi(x) \cap \Phi(y_i)$. Since $\beta \in \Phi(x) \cap \Phi(y_i)$, this then implies that $|\Phi(x) \cap \Phi(y_i)| > 1$, a contradiction. 
    \par
    Now modify $\pi$ and replace the subpath $(y_{i-1}, y_i, y_{i+1})$ with the other embedded edge-path in the residue $R$ from $y_{i-1}$ to $y_{i+1}$. By \cite[Theorem 2.9]{Ronan_2009} all the elements of $R$ are $\preceq y_i$, since the element Proj$_{R}(e)$ of $R$ is opposite to $y_i$. It then follows by  \cite[Theorem 2.15]{Ronan_2009} that all elements of $R$ lie on a geodesic edge path from Proj$_{R}(e)$ to $y_i$ and hence on a geodesic edge path from $e$ to $y_i$ through Proj$_{R}(e)$. Thus this path modification decreases the complexity of $\pi$ (under lexicographic order); considering $\pi$ as the tuple $\Pi = (n_L, \ldots , n_2, n_1)$, where $n_j$ is the number of elements $p_i$ in $\pi$ with $\ell(p_i) = j$ and $L = max_{i = 0}^{n} \ell(p_i)$. 
    \par
    After finitely many such modifications we obtain the desired path. The fact that the minimal element $y_{\beta}$ is necessarily a tight gate follows by \Cref{minimal_witness} and convexity follows by \Cref{witnesses_are_convex}.
\end{proof}

\begin{corollarySec} \label{cor:same_witness_for_all_elements}
    For each $T \in \mathbb{T}$ and $\beta \in \partial T$, there is a unique minimal length element $y_{\beta} \in W$ such that 
    \begin{equation*}
        \Phi(x) \cap \Phi(y_{\beta}) = \{\beta \}
    \end{equation*}
    for all $x \in Q(T)$
\end{corollarySec}
\begin{proof}
    Let $u,v \in Q(T)$ and $\beta \in \partial T$. By \Cref{thm:witnesses_are_gated} there are unique minimal length elements $u', v'$ such that $\Phi(u) \cap \Phi(u') = \{\beta \}$ and $\Phi(v) \cap \Phi(v') = \{\beta \}$. Let $g_T$ be the minimal element of $Q(T)$. Then $g_T \preceq u$ and $g_T \preceq v$ and $\Phi(g_T) \cap \Phi(u') = \{\beta \}$ and $\Phi(g_T) \cap \Phi(v') = \{\beta \}$. If $u' \neq v'$ then applying \Cref{thm:witnesses_are_gated} to $g_T$, there is $z \in W$ with $z \prec u'$ and $z \prec v'$ such that $\Phi(g_T) \cap \Phi(z) = \{\beta \}$. But then also $\Phi(u) \cap \Phi(z) = \{\beta \}$ and $\Phi(v) \cap \Phi(z) = \{\beta \}$ contradicting the minimality of $u'$ and $v'$. Hence we must have $u' = v'$.
\end{proof}

As a consequence of \Cref{thm:witnesses_are_gated} we also obtain the following properties for $\partial T_{\beta}$.

\begin{corollarySec} \label{cor:properties_of_witnesses_conetype_level}
    For each $T \in \mathbb{T}$ and $\beta \in \partial T$, the following properties hold:
    \begin{enumerate}[(i)]
        \item $\partial T_{\beta} \neq \emptyset$.
        \item $\partial T_{\beta} = \bigcap_{y \in Q(T)} \partial T(y^{-1})_{\beta}$.
        \item $\partial T_{\beta}$ is convex and gated.       
    \end{enumerate}
\end{corollarySec}
\begin{proof}
    \begin{enumerate}[(i)]
        \item This directly follows from \Cref{cor:same_witness_for_all_elements}.
        \item Follows by definition of $\partial T_{\beta}$.
        \item This follows by (i) and (ii) and by the fact that each $\partial T(y^{-1})_{\beta}$ is convex. Moreover, by \Cref{cor:same_witness_for_all_elements} the gate $y_{\beta}$ is independent of $y \in Q(T)$, so the intersection remains gated with gate $y_{\beta}$.
    \end{enumerate}
\end{proof}

The results of this section show that for any cone type $T$ and any boundary root $\beta \in \partial T$, there exists a unique tight gate $y_\beta$ of minimal length such that $\Phi(x)\cap\Phi(y_\beta)=\{\beta\}$ for all $x\in Q(T)$.
Moreover, by symmetry, there exists a unique tight gate $x_\beta \preceq x$ of minimal length satisfying $\Phi(x_\beta)\cap\Phi(y_\beta)=\{\beta\}$. Thus each positive root $\beta$ that arises as a boundary root of some cone type
canonically determines an unordered pair $\{x_\beta,y_\beta\}$ of tight gates associated to the root $\beta$.

\begin{definition} \label{def:super_elementary_roots}
    A root $\beta \in \Phi^+$ is \textit{super-elementary} if there exists $x,y \in W$ with $\Phi(x) \cap \Phi(y) = \{ \beta \}$. We denote the set of super-elementary roots by $\mathcal{S}$.
\end{definition}

By the finiteness of the set $\mathbb{T}$, the set $\mathcal{S}$ is finite for every finitely generated Coxeter group. The following example demonstrates that there are potentially multiple pairs of minimal-length tight gates $x,y$ with $\Phi(x) \cap \Phi(y) = \{ \beta \}$ and $\Phi^R(x) = \Phi^R(y) = \{ \beta \}$ for a super-elementary root $\beta$.

\begin{example} \label{ex:multiple_pairs_of_tight_gates_for_boundary_root}
    Let $W$ be the rank $4$ compact hyperbolic Coxeter group whose graph is given by
    \begin{figure}[H]
        \centering
        \begin{tikzpicture}
    \node (s) at (0, 1.5) {$s$};
    \node (t) at (1.5, 1.5) {$t$};
    \node (u) at (1.5, 0) {$u$};
    \node (v) at (0, 0) {$v$};
    
    \draw (s) -- (t);
    \draw (t) -- (u);
    \draw (u) -- (v) node[midway, above] {4};
    \draw (v) -- (s);
\end{tikzpicture}
    \end{figure}

    Let $\beta = \alpha_t + \sqrt{2}\alpha_u + 2\alpha_v$. Denote the elements $a = tvutv, b = uvutv, c = utvutv, d = vutv$. Then these elements are all tight gates with $\Phi^R(a) = \Phi^R(b) = \Phi^R(c) = \Phi^R(d) = \{ \beta \}$ and we have
    \begin{align*}
        &\ \Phi(a) \cap \Phi(b) = \{ \beta \} \\
        &\ \Phi(c) \cap \Phi(d) = \{ \beta \} 
    \end{align*}
    But $d$ is not a witness of $\beta$ with respect to $a$ and $b$ is not a witness of $\beta$ with respect to $c$.   
\end{example}

In subsequent work, we investigate classes of Coxeter groups in which every
super-elementary root $\beta \in \mathcal{S}$ admits a unique pair of tight
gates, up to order, whose inversion sets intersect precisely in $\{\beta\}$.

The following result implies that the set of tight gates $\Gamma^0$ is closed under suffix and gives more precise information about the relationship between tight gates and super-elementary roots.

\begin{proposition} \label{suffix_closure_of_tight_gates}
    Let $x \in \Gamma$ and $\Phi(x) \cap \Phi(y) = \{ \beta \}$ and $\Phi^R(y) = \{ \beta \}$ where $y$ is of minimal length with this property. If $s \in D_L(y)$ then 
    \begin{equation*}
        \Phi(sx) \cap \Phi(sy) = \{ s\beta \}
    \end{equation*}
    and $sy \in \Gamma^0$.
\end{proposition}
\begin{proof}
Let $x, y \in W$ possess the properties as stated. Since $\alpha_s \in \Phi(y)$ it follows that $\alpha_s \notin \Phi(x)$. Denote $v = sx$. It then follows that $\Phi(v) = \{ \alpha_s \} \sqcup s\Phi(x)$ and $\Phi(sy) = s\Phi(y) \setminus \{ \alpha_s \}$. Therefore $\Phi(v) \cap \Phi(sy) = \{ s\beta \}$ and it is clear that $s\beta \in \Phi^R(sy)$. Now suppose there is $\alpha \neq s\beta \in \Phi^R(sy)$. Then we have $s_{\alpha} sy \preceq sy$ and $\alpha_s \notin \Phi(s_{\alpha}sy) \cap \Phi(sy)$. Thus $ss_{\alpha}sy \preceq y$ with $\ell(s_{\alpha}sy) = \ell(sy) -1 = \ell(y) - 2$ and $s\beta \in \Phi(s_{\alpha}sy)$. Then $\Phi(ss_{\alpha}sy) = \{ \alpha_s \} \sqcup \Phi(s_{\alpha}sy)$ and $\beta \in \Phi(ss_{\alpha}sy)$. This then implies that 
\begin{equation*}
    \Phi(x) \cap \Phi(ss_{\alpha}sy) = \{ \beta \}
\end{equation*}
and $\ell(ss_{\alpha}sy) = \ell(y) -1$ contradicting the minimality of $y$.
\end{proof}

The following two additional results demonstrate how boundary roots and the set of witnesses can be iteratively computed.

\begin{lemma}\cite[Lemma 3.3.4]{Yau21} \label{boundary_roots_of_ws_from_boundaryroots_of_w}
Let $w \in W$ and $s \in S$ where $s \notin D_L(w)$. Then
\begin{equation*}
    \partial T (w^{-1}s) = \{ \alpha_s \} \sqcup s \big( \{ \alpha \in \partial T(w^{-1}) \mid \exists \hspace{0.1cm} x \in \partial T(w^{-1})_{\alpha} \mbox{ with } s \in D_L(x) \} \big)
\end{equation*}
\end{lemma}

\begin{proof}
We first show the left to right inclusion. First, $\Phi(sw) \cap \Phi(s) = \{ \alpha_s \}$ so $\alpha_s \in \partial T (w^{-1}s)$. If $\alpha \neq \alpha_s \in \partial T(w^{-1}s)$ then by \Cref{thm:boundary_roots_definition} there is $x\in W$ such that $\Phi(sw) \cap \Phi(x) = \{ \alpha \}$. We claim that $s\alpha \in \partial T(w^{-1})$ and that $sx$ is a witness of $s\alpha$ with respect to $w$. Since $\alpha_s \in \Phi(sw)$ we have $\alpha_s \notin \Phi(x)$ and hence $\Phi(sx) = \{ \alpha_s \} \sqcup s \Phi(x)$. Also, $\Phi(w) = s \big(\Phi(sw) \setminus \{ \alpha_s \} \big)$. Therefore
\begin{align*}
    \Phi(w) \cap \Phi(sx) =& \ s \big(\Phi(sw) \setminus \{ \alpha_s \} \big) \cap \big(  \{ \alpha_s \} \sqcup s \Phi(x) \big)\\
    =& \ s \big( \Phi(sw) \cap \Phi(x) \big) \\
    =& \ s(\{ \alpha \})
\end{align*}
Thus the claim follows. We now show the right to left inclusion. For a root $\alpha \in \partial T(w^{-1})$ where there is $x \in W$ with $\Phi(w) \cap \Phi(x) = \{ \alpha \} \mbox{ and } s \in D_L(x)$ we aim to show that $\Phi(sw) \cap \Phi(sx) = \{ s\alpha \}$. It then follows by \Cref{thm:boundary_roots_definition} again that $s\alpha$ is a boundary root of $T(w^{-1}s)$. Since $\ell(sw) > \ell(w)$ we have that $\Phi(sw) = \{ \alpha_s \} \sqcup s\Phi(w)$. If $x \in W$ is such that $s \in D_L(x)$ then $\Phi(sx) = s \big( \Phi(x) \setminus \{ \alpha_s \} \big)$. Hence we have
\begin{equation*}
    \Phi(sw) \cap \Phi(sx) = \big( \{ \alpha_s \} \sqcup s\Phi(w) \big) \cap s \big( \Phi(x) \setminus \{ \alpha_s \} \big)
\end{equation*}
Since $\Phi(w) \cap \Phi(x) = \{ \alpha \}$ by the formula for $\Phi(sw) \cap \Phi(sx)$ above it follows that $s\alpha \in \Phi(sw) \cap \Phi(sx)$. We now show that $\Phi(sw) \cap \Phi(sx) = \{ s\alpha \}$. If $ \beta \in \Phi(sw) \cap \Phi(sx)$ then $\beta \in s\Phi(w) \cap s\Phi(x) = s \big(\Phi(w) \cap \Phi(x) \big) = s( \{ \alpha \} )$. So $\beta = s\alpha$ as required.
\end{proof}

The next result shows how the set $\partial T(sx^{-1})_{s\beta}$ is obtained from $\partial T(x^{-1})_{\beta}$ if $\beta \in \partial T(x^{-1})$ and $s\beta \in \partial T(sx^{-1})$.

\begin{corollarySec} \label{witness_evolution}
    Let $x \in W$, $s \in S$ with $s \notin D_L(x)$. Let $y = sx$. If $\beta \in \partial T(x^{-1})$ and $s\beta \in \partial T(y^{-1})$ then 
    \begin{equation*}
        \partial T(y^{-1})_{s\beta} = \{ w \in \partial T(x^{-1})_{\beta} \mid s \in D_L(w) \}
    \end{equation*}
\end{corollarySec}

\begin{proof}
    If $v \in \partial T(y^{-1})_{s\beta} $ then $\Phi(y) \cap \Phi(v) = \{ s\beta \}$. Since $\alpha_s \in \Phi(y)$ this implies that $\alpha_s \notin \Phi(v)$ and so $\Phi(sv) = \{ \alpha_s \} \sqcup s\Phi(v)$. Since $\Phi(x) = s\Phi(y) \setminus \{ \alpha_s \}$ we have
    \begin{equation*}
        \Phi(x) \cap \Phi(sv) = \big( s\Phi(y) \setminus \{ \alpha_s \} \big) \cap \big( \{ \alpha_s \} \sqcup s\Phi(v) \big)
    \end{equation*}
    Therefore, as in the proof of \Cref{boundary_roots_of_ws_from_boundaryroots_of_w}, we have $\Phi(x) \cap \Phi(sv) = \{ \beta \}$.
    The reverse inclusion directly follows from the second half of the proof of \Cref{boundary_roots_of_ws_from_boundaryroots_of_w}.
\end{proof}

\section{The $\mathscr{T}^0$ partition and join-irreducible elements of $\Gamma$} \label{sec:Tight gates and Join-irreducible elements of  Gamma}

We begin this section by describing a natural join representation for elements in Coxeter groups, similar to that of Reading and Speyer's \textit{canonical join representation} (see \Cref{sec:canonical_join_representations} for further discussion). We use the join representation studied here to give additional properties of the tight gates $\Gamma^0$ of $\mathscr{T}$ and later use this join representation to give an alternative proof of Reading and Speyer's result.

We then show that the partition induced by the tight gates $\Gamma^0$ produces the cone type arrangement $\mathscr{T}$. We also show that the set $\Gamma^0$ is the set of join-irreducible elements of the poset $\Gamma$ under the right weak order. In particular, it follows from \cite[Theorem 1.4]{ALCO_2024__7_6_1879_0} that the set $\Gamma^0$ are the most fundamental elements of the smallest Garside shadow of $W$.

\begin{lemma} \label{prefixes_are_convex}
    Let $x\in W$. Then the set of prefixes of $x$, $P(x) = \{z \preceq x \mid z \in W\}$ is convex.
\end{lemma}
\begin{proof}
    Clearly, we have $P(x) = \bigcap_{\Phi^{+} \setminus \Phi(x)} H_{\alpha}^{+}$. Hence the result follows by \Cref{lem:convexity}. 
\end{proof}

We first give a few preliminary results. The following result is an obvious, but rather useful fact.
\begin{lemma} \label{tight_prefixes}
    Let $x \in W$. If $\alpha \in \Phi(x)$ then there is $y \preceq x$ with $\Phi^R(y) = \{ \alpha \}$
\end{lemma}
\begin{proof}
    Let $x_0$ be a prefix of $x$ with $\alpha \in \Phi^R(x_0)$. If there is $\beta \in \Phi^R(x_0)$ with $\beta \neq \alpha$ then consider the prefix $x_1 := s_{\beta}x_0$ of $x$. Repeating this process, since each step replaces $x_i$ by a shorter prefix $s_{\beta}x_i$, the process terminates and there must be some prefix $x_k$ of $x$ with $\Phi^R(x_k) = \{ \alpha \}$.
\end{proof}

\begin{proposition} \label{prop:x_is_join_of_phi1_tight_elements} 
	Let $x \in W$. Then 
	\begin{equation*}
		x = \bigvee \phi^1(x)
	\end{equation*}
	where $\phi^1(x) = \{ z \in W \mid z \preceq x \text{ with } \Phi^R(z) = \{ \alpha \} \text{ and } \alpha \in \Phi^1(x) \}$.
\end{proposition}
\begin{proof}
	Clearly $x$ is an upper bound of $\phi^1(x)$. So by \cite[Theorem 3.2.1]{BB05} the join of $\phi^1(x)$ exists. Let $y = \bigvee \phi^1(x)$. Then $y \preceq x$ with $\Phi(y) \subseteq \Phi(x)$. By \Cref{inversion_set_of_join} we have
	\begin{equation*}
		\Phi(y) = cone_{\Phi}\big(\bigcup_{z \in \phi^1(x)} \Phi(z) \big) = cone_{\Phi}\big(\bigcup_{z \in \phi^1(x)} cone_{\Phi}\Phi^1(z) \big) = cone_{\Phi}\big( \bigcup_{z \in \phi^1(x)} \Phi^1(z) \big)
	\end{equation*}
	By definition of $\phi^1(x)$, if $\alpha \in \Phi^1(x)$ then $\alpha \in \Phi^R(z) \subseteq \Phi^1(z)$ for some $z \in \phi^1(x)$, so $\Phi^1(x) \subseteq \bigcup_{z \in \phi^1(x)} \Phi^1(z)$. Therefore,
	\begin{equation*}
		\Phi(x) = cone_{\Phi}\big( \Phi^1(x) \big) \subseteq cone_{\Phi}\big( \bigcup_{z \in X} \Phi^1(z) \big) = \Phi(y)
	\end{equation*}
	and $x \preceq y$. Hence $x = y$.
\end{proof}

We note that nothing in \Cref{prop:x_is_join_of_phi1_tight_elements} suggests that for each $\alpha \in \Phi^1(x)$ there cannot be more than one prefix $z$ of $x$ with $\Phi^R(z) = \{\alpha\}$. The following result shows that there is a bijection between the set $\phi^1(x) = \{ z \in W \mid z \preceq x \text{ with } \Phi^R(z) = \{ \alpha \} \text{ and } \alpha \in \Phi^1(x) \}$ and $\Phi^1(x)$.

\begin{theorem} \label{prop:bijection_join_refine_phi1}
    Let $x \in W$. For each $\beta \in \Phi^1(x)$ there is a unique minimal length $z \preceq x$ such that $\Phi^R(z) = \{ \beta \}$. Thus the map $\Phi_R: \phi^1(x) \rightarrow \Phi^1(x)$ defined by $z \mapsto \Phi^R(z)$ is a bijection.
\end{theorem}
\begin{proof}
    The proof is very similar to the proof of \Cref{thm:witnesses_are_gated}. Let $\beta \in \Phi^1(x)$ and denote $Z_{\beta} = \{z \in W \mid z \preceq x \text{ and } \beta \in \Phi(z)\}$. We show that for each $z_0, z_n \in Z_{\beta}$ there is $z \in Z_{\beta}$ such that $z_0 \succeq z  \preceq z_n$. Let $\pi = (z_0, z_1 , \ldots , z_n)$ be a geodesic edge path between $z_0$ and $z_n$. Note that in terms of half-spaces
    $$Z_{\beta} = \big( \bigcap_{\Phi^{+} \setminus \Phi(x)} H_{\alpha}^{+} \big) \cap H^{-}_{\beta} $$
    and hence by \Cref{lem:convexity} $Z_{\beta}$ is convex. Therefore, each $z_i \in \pi$ is also in $Z_{\beta}$. Thus $H_{\beta}$ does not separate any $z_i$ from $z_j$ in $\pi$ and each $z_i$ is a prefix of $x$. 
    \par
    Our aim is to modify $\pi$ to obtain another embedded edge path between $z_0$ and $z_n$ so that there is no $z_{i-1} \prec z_i \succ z_{i+1}$. Now suppose there is $z_{i-1} \prec z_i \succ z_{i+1}$. Let $H_r, H_q$ be the walls separating $z_{i-1}$ from $z_i$ and $z_i$ from $z_{i+1}$ respectively and let $\alpha_r, \alpha_q$ be the corresponding roots. Now $z_{i-1}, z_{i+1}$ are elements of the \textit{residue} $R = \langle r, q \rangle (z_i)$ and $\alpha_r, \alpha_q \in \Phi^R(z_i)$. Hence by \cite[Theorem 2.9]{Ronan_2009} we have $y \preceq z_i$ for all $y \in R$ and as a consequence we have that $y \preceq x$ for all $y \in R$. If $r$ and $q$ do not commute then $r,q$ are sharp angled and the identity $e$ is in a geometric fundamental domain for the finite dihedral reflection subgroup $W_{\langle r, q \rangle}$ generated by the reflections $r,q$. We claim that all the elements of the residue $R = \langle r, q \rangle (z_i)$ lie in $Z_{\beta}$. 
    \par
    We need to show that $ \beta \notin \Phi^{+}_{\langle r, q \rangle}$. If $r$ and $q$ commute then we are done since there is no other root in $\Phi^{+}_{\langle r, q \rangle}$, otherwise if $ \beta \in \Phi^{+}_{\langle r, q \rangle}$ then since $\beta = k \alpha_r + j \alpha_q$ for some $k,j \in \mathbb{R}_{> 0}$ and $\alpha_r, \alpha_q \in \Phi(z_i) \subseteq \Phi(x)$ this implies $\beta \in cone_{\Phi}(\Phi(x) \setminus \{\beta\})$, contradicting \Cref{inversion_set_is_determined_by_phi1}.
    \par
    We now modify $\pi$ and replace the subpath $(z_{i-1}, z_i, z_{i+1})$ with the other embedded edge-path in the residue $R$ from $z_{i-1}$ to $z_{i+1}$. This path modification decreases the complexity of $\pi$ (under lexicographic order); considering $\pi$ as the tuple $\Pi = (n_L, \ldots , n_2, n_1)$, where $n_j$ is the number of elements $z_i$ in $\pi$ with $\ell(z_i) = j$ and $L = max_{i = 0}^{n} \ell(z_i)$. 
    \par
    After finitely many such modifications we obtain the desired path. The fact that the minimal element $z$ necessarily has $\Phi^R(z) = \{ \beta \}$ follows by \Cref{tight_prefixes}.
\end{proof}

We briefly reall some additional definitions and terminology before proceeding further. Let $(L, \le)$ be a poset. A \textit{join representation} for $x \in L$ is an identity of the form $x = \bigvee U$ for some $U \subseteq L$. A join representation $x = \bigvee U$ is \textit{irredundant} if there does not exist a proper subset $U' \subset  U$ such that $x = \bigvee U'$.  
For subsets $U, V$ of $L$, denote $U << V$ if for every $u \in U$ there exists $v \in V$ such that $u \le v$. The relation $U << V$ is called a \textit{join-refinement} since if $U << V$ then $\bigvee U \le \bigvee V$. The next result shows that the set $\phi^1(x)$ \textit{join-refines} every join-representation $Y$ of $x$.

\begin{corollarySec} \label{join_refinements}
    Let $x \in W$ and suppose $x = \bigvee Y$ for $Y \subseteq W$. Then for every $z \in \phi^1(x)$ there is $y \in Y$ such that $z \preceq y$.
\end{corollarySec}
\begin{proof}
    Since $x = \bigvee Y$, by \Cref{inversion_set_of_join} we have
    \begin{align*}
        \Phi(x) &= \ cone_{\Phi}\big( \bigcup_{y \in Y} \Phi(y) \big) \\ 
        &= \ cone_{\Phi}(\bigcup_{y \in Y} cone_{\Phi}(\Phi^1(y)) \\
        &= \ cone_{\Phi} \big(\bigcup_{y \in Y} \Phi^1(y) \big)
    \end{align*}
 Again by \Cref{inversion_set_is_determined_by_phi1} we have $\Phi^1(x) \subseteq \bigcup_{y \in Y} \Phi^1(y)$ and thus for $\alpha \in \Phi^1(x)$, we have $\alpha \in \Phi^1(y)$ for some $y \in Y$. Hence by \Cref{prop:bijection_join_refine_phi1} there is $z \preceq y \preceq x$ with $\Phi^R(z) = \{ \alpha \}$. Then again by \Cref{prop:bijection_join_refine_phi1}, $z$ is the unique element in $\phi^1(x)$ such that $\Phi^R(z) = \{ \alpha \}$ and the result follows.
\end{proof}

Our next result shows that every gate of $\mathscr{T}$ is a join of tight gates.

\begin{theoremSec} \label{thm:every_gate_is_a_join_of_tightgates}
    Let $x \in \Gamma$. Then 
    \begin{equation*}
       x  = \bigvee \partial t(x^{-1})
    \end{equation*}
    where $\partial t(x^{-1}) = \{ z \in W \mid z \preceq x \text{ with } \Phi^R(z) = \{ \alpha \} \text{ and } \alpha \in \partial T(x^{-1})\}$. Hence, every gate is a join of tight gates.
\end{theoremSec}
\begin{proof}
    Let $x \in \Gamma$. By \Cref{tight_prefixes}, for each $\alpha \in \partial T(x^{-1})$, there is $z_{\alpha} \preceq x$ with $\Phi^R(z_{\alpha}) = \{ \alpha \}$. Since $\alpha \in \partial T(x^{-1})$, by \Cref{thm:boundary_roots_definition} and \Cref{minimal_witness} there exists $w_{\alpha} \in \Gamma^0$ such that $\Phi(x) \cap \Phi(w_{\alpha}) = \{ \alpha \}$. Therefore, we also have $\Phi(z_{\alpha}) \cap \Phi(w_{\alpha}) = \{ \alpha \}$. By \Cref{thm:gates_characterised_by_phi0} it follows that $z_{\alpha} \in \Gamma^0$ and $x$ is an upper bound of 
    \begin{equation*} \label{tight_gate_prefixes}
    \partial t(x^{-1}) := \{ z \in W \mid z \preceq x \text{ with } \Phi^R(z) = \{ \alpha \} \text{ and } \alpha \in \partial T(x^{-1})\} \subset \Gamma^0
    \end{equation*}
    Now let $y = \bigvee \partial t(x^{-1}) $ and let $\Phi(\partial t(x^{-1}) ) := \bigcup_{z \in \partial t(x^{-1}) } \Phi(z)$. By \Cref{cone_type_of_join} and \Cref{thm:cone_type_formulas} we have 
    \begin{equation*}
    T(y^{-1}) = \bigcap_{z \in \partial t(x^{-1}) } T(z^{-1}) = \bigcap_{\Phi(\partial t(x^{-1}) )} H_{\beta}^+   
    \end{equation*}
    Since $y \preceq x$ and $\partial T(x^{-1}) \subseteq \Phi(\partial t(x^{-1}) )$ it follows by \Cref{lem:suffix_conetype_containment} and  \Cref{thm:cone_type_formulas} again that
    \begin{equation*}
        T(x^{-1}) \subseteq T(y^{-1}) = \bigcap_{\Phi(\partial t(x^{-1}) )} H_{\beta}^+ \subseteq \bigcap_{\partial T(x^{-1})} H_{\beta}^{+} = T(x^{-1})
    \end{equation*}
    So $T(x^{-1}) = T(y^{-1})$ and since $x \in \Gamma$, by \Cref{thm:gates_of_conetype_partition} we must have $y = x$.
\end{proof}

\begin{corollarySec} \label{cor:partial_elements_contained_in_phi1_elements}
    For $x \in W$, $\partial t(x^{-1}) \subseteq \phi^1(x)$.
\end{corollarySec}
\begin{proof}
    \Cref{prop:bijection_join_refine_phi1} shows that for each $\beta \in \Phi^1(x)$ there is a unique minimal length element $j_{\beta} \in \phi^1(x)$. Since $\partial T(x^{-1}) \subseteq \Phi^1(x)$ the result follows by \Cref{thm:every_gate_is_a_join_of_tightgates} and \Cref{prop:x_is_join_of_phi1_tight_elements}.
\end{proof}

\begin{corollarySec} \label{cor:every_conetype_is_expressible_as_an_intersection_of_tightgates}
    Each cone type $T \in \mathbb{T}$ is expressible as an intersection of cone types of tight gates.
\end{corollarySec}
\begin{proof}
    By \Cref{thm:every_gate_is_a_join_of_tightgates} it follows that every gate is a join of a set of tight gates $\partial t(x^{-1}) $. Then by \Cref{cone_type_of_join} we have that $T(x^{-1}) = \bigcap_{z \in \partial t(x^{-1}) } T(z^{-1})$.
\end{proof}

\begin{corollarySec} \label{cor:tight_gate_arrangement}
	 Define $x \sim y$ if the following holds: $x \in T(w^{-1})$ if and only if $y \in T(w^{-1})$ for all $w \in \Gamma^0$. 
	 \par
	 Let $\mathscr{T}^0$ be the partition of $W$ with parts $P$ defined by the relation $x \sim y$. Then 
	 $\mathscr{T}^0 = \mathscr{T}$ and hence $\mathscr{T}^0$ is a regular partition.
\end{corollarySec}
\begin{proof}
	By \Cref{thm:conetypepartition_asintersectionofconetypes} it suffices to show that if $x,y$ are in the same part of $\mathscr{T}^0$ then  $x,y$ are in the same part of $\mathscr{T}$. Suppose $x \in T(z^{-1})$ if and only if $y \in T(z^{-1})$ for all $z \in \mathscr{T}^0$. Let $x \in T(g^{-1})$ for $g \in \Gamma \setminus \Gamma^0$. By \Cref{thm:every_gate_is_a_join_of_tightgates} and \Cref{cone_type_of_join} we then have $x \in T(w^{-1})$ for all $w \in \partial t(g^{-1})$ where 
	$$ \partial t(g^{-1})  = \{ w \in W \mid w \preceq g \text{ with } \Phi^R(w) = \{ \alpha \} \text{ and } \alpha \in \partial T(g^{-1})\}$$ 
	Since $\partial t(g^{-1}) \subset \Gamma^0$ this implies $y \in T(w^{-1})$ for all $w \in \partial t(g^{-1})$ and hence by \Cref{thm:every_gate_is_a_join_of_tightgates} and \Cref{cone_type_of_join} again this implies $y 
	\in T(g^{-1})$. Similarly, if $x \notin T(g^{-1})$ for some $g \in \Gamma \setminus \Gamma^0$ then by  \Cref{thm:every_gate_is_a_join_of_tightgates}, there must be some $w \preceq g$ with $\Phi^R(w) = \{ \alpha \}$ and $\alpha \in \partial T(g^{-1}) \cap \partial T(w^{-1})$ such that $x \notin T(w^{-1})$. Therefore, also $y \notin T(w^{-1})$.
\end{proof}

\begin{figure}
	\includegraphics[scale=0.3]{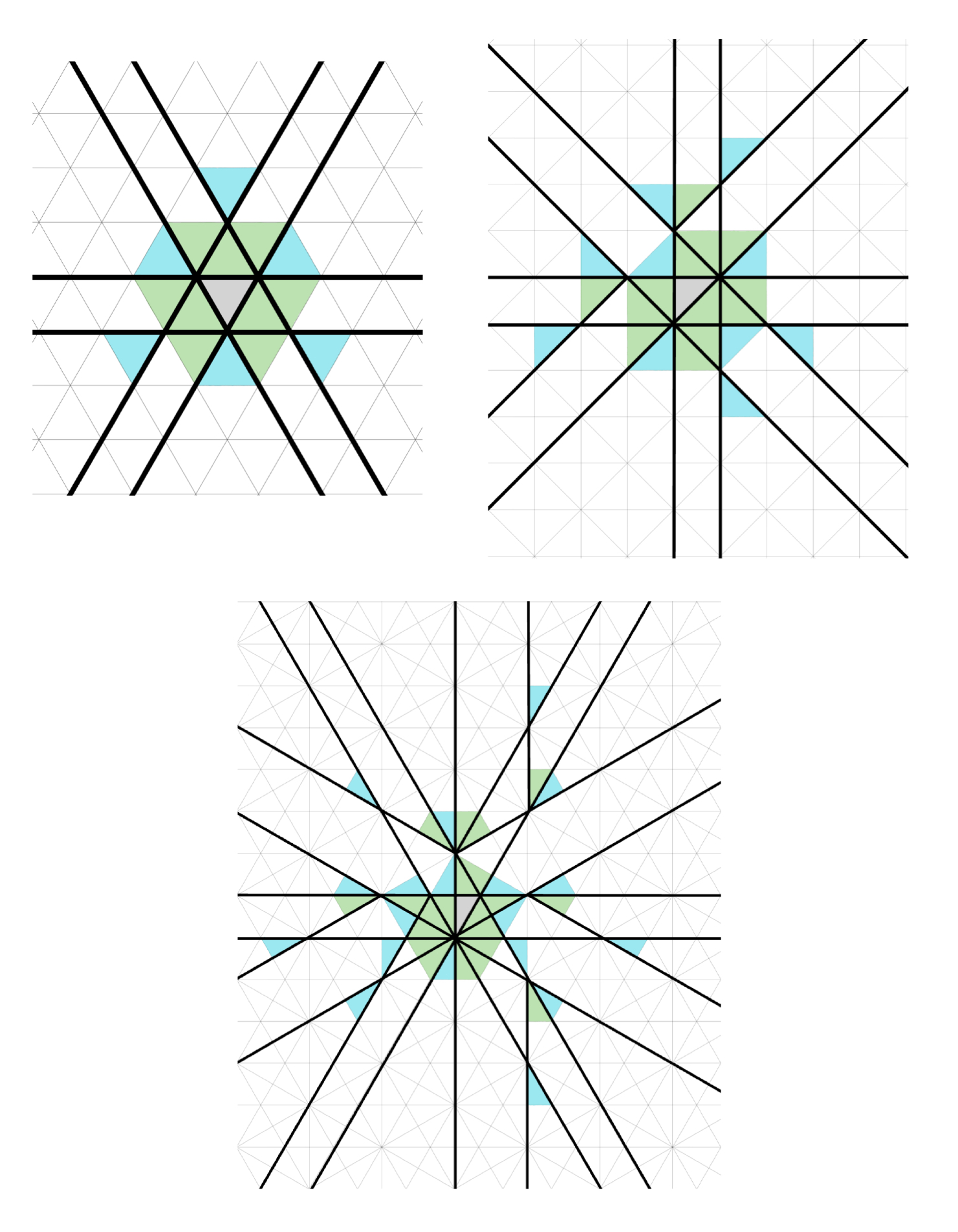}
	\caption{The partition $\mathscr{T}^0 = \mathscr{T}$ for the rank 3  Coxeter groups of affine type ($\widetilde{A}_2$, $\widetilde{B}_2$ and $\widetilde{G}_2$). The green alcoves represent the tight gates and the cyan alcoves are the gates which are non-trivial joins of tight gates.}
	\label{fig:tight_gate_arrangement_affine_rank_3}
\end{figure}

\begin{remark}
	A consequence of \Cref{cor:tight_gate_arrangement} is that to determine whether two elements $x,y$ have the same cone type (or equivalently, whether $x^{-1}$ and $y^{-1}$ live in the same part of the cone type partition), one only needs to compute the set of tight gates $\Gamma^0$ and then check whether  $x \in T(w^{-1})$ if and only if $y \in T(w^{-1})$ for all $w \in \Gamma^0$. Our computations using Sagemath (\cite{sagemath}) in \Cref{subsec:data_for_select_coxeter_systems} show that often, the set $\Gamma^0$ is much smaller than $\Gamma$. We illustrate the $\mathscr{T}^0$ arrangements for the rank 3 Coxeter groups of affine type in \Cref{fig:tight_gate_arrangement_affine_rank_3}.
\end{remark}

Recall that for a poset $P$ an element $x \in P$ is \textit{join-irreducible} if $x$ is not the minimal element in $P$ and if $x = a \vee b$ then either $x = a$ or $x = b$. Equivalently, if $X \subseteq P$ and $x = \bigvee X$ then $x \in X$.

\begin{theoremSec} \label{join_irreducible_theorem}
	The set 
	$$\Gamma^0 = \{x \in \Gamma \mid |\Phi^R(x)| = 1 \}$$ 
	is the set of join-irreducible elements of $(\Gamma , \preceq)$.
\end{theoremSec}
\begin{proof}
    Let $x \in \Gamma^0$ and $\Phi^R(x) = \{ \beta \}$. If $x$ is not join-irreducible, then $x  = \bigvee A$ for some $A \subseteq \Gamma$ with $x \notin A$. Since $\beta \in \Phi^1(x)$ by \Cref{cor:phi1_of_join_isphi1_of_x} $\beta \in \Phi^1(y)$ for some $y \in A$ and $y \prec x$. Therefore, there must be $\alpha \neq \beta \in \Phi^R(x)$, a contradiction.
    \par
    Now let $x \in \Gamma$ be join-irreducible in $(\Gamma, \preceq)$ and suppose $|\Phi^R(x)| > 1$. Let $\Phi^R(x) = \{ \alpha_1 , \ldots \alpha_k \}$ and $x_i := s_{\alpha_{i}} x$. Let $\pi(x_i) \in \Gamma$ be the minimal element of the cone type part containing $x_i$, that is the part $Q(T(x_{i}^{-1}))$. We claim that $x =  \bigvee \pi(x_i)$. Let $y = \bigvee \pi(x_i)$. Note that since $x$ is an upper bound for each $\pi(x_i)$, this implies $y \preceq x$. Then since $x = \bigvee x_i$ and $T(x_{i}^{-1}) = T(\pi(x_{i})^{-1})$ for each $i$, by \Cref{cone_type_of_join} we have
    \begin{equation*}
    T(x^{-1}) = \bigcap_{x_i} T(x_{i}^{-1}) 
    		 = \bigcap_{x_i} T(\pi(x_{i})^{-1})
		 = T(y^{-1})
    \end{equation*}
    Since $x \in \Gamma$ we must have $x = y$.
\end{proof}

\section{Computing the sets $\mathcal{S}$ and $\Gamma^0$} \label{sec:computing_gamma0}

The super elementary roots $\mathcal{S}$ and tight gates $\Gamma^0$ are intimately connected by definition. In this section, we utilise the results of \Cref{sec:Tight gates and Join-irreducible elements of  Gamma} to give a direct and efficient method for computing $\Gamma^0$ without computing $\Gamma$ first; effectively obtaining the minimal data to determine whether two elements have the same cone type. 

Let us first give a bit more background to set up the algorithm and illustrate the connection between our work and some important recent developments in Coxeter group theory.

We first recall the definition of the set of \textit{elementary} roots $\mathscr{E}$ in Coxeter groups. For roots $\alpha, \beta \in \Phi^{+}$ we say $\beta$ \textit{dominates} $\alpha$ if whenever $w\beta < 0$ implies $w\alpha < 0$ for all $w \in W$. A root is \textit{elementary} if it does not dominate any root besides itself. Introduced by Brink and Howlett in \cite{BH93}, the set of elementary roots $\mathscr{E}$ were shown to be finite for any finitely generated Coxeter group \cite[Theorem 2.8]{BH93} and play a fundamental role in their finite state automaton recognising the language of reduced words of $(W,S)$. They have been extensively studied (see for example \cite[Chapter 4]{BB05}, \cite{HD16}, \cite{dyer2024shi}) and are computed in a straightforward manner, as described in the aforementioned work. For our purposes, it is enough to know that the set $\mathscr{E}$ can be easily computed.

The \textit{Shi arrangement} or \textit{Shi partition} $\mathscr{S}$ is the partition of $W$ induced by the elementary walls, and like the cone type partition $\mathscr{T}$, each part corresponds to a state of a finite state automaton recognising the language of reduced words of $(W,S)$ (see the example of $\widetilde{A}_2$ in \Cref{fig:tight_gate_arrangement_affine_rank_3}, in this case $\mathscr{S} = \mathscr{T}$). This is in fact one way to visualise Brink and Howlett's automaton. Also, like the cone type partition $\mathscr{T}$, each part $P$ of $\mathscr{S}$ contains an element of minimal length, which is a gate of $P$ (proven in \cite[Theorem 1.1]{dyer2024shi} and \cite[Theorem 1.1]{ALCO_2024__7_6_1879_0}). The unique minimal length representatives of these parts are known as the \textit{low}-elements $L$. Formally, we have the following definition of low-elements.

\begin{definition} \cite[Definition 3.24]{HD16}
    An element $w  \in W$ is \textit{low} if $\Phi(w) = cone_{\Phi}(A)$ for some $A \subseteq \mathscr{E}$. Equivalently, $w$ is \textit{low} if $\Phi^1(w) \subseteq \mathscr{E}$.
\end{definition}

For the purposes of \Cref{algo:compute_tight_gates_and_super_elementary_roots} we pay particular attention to the set of \textit{tight} low elements, that is the set 
\begin{equation*}
    L^0 = \{w  \in W \mid |\Phi^R(w)| = 1 \text{ and } \Phi^R(w) \subseteq \mathscr{E} \}
\end{equation*}

Crucially for \Cref{algo:compute_tight_gates_and_super_elementary_roots}, by \cite[Theorem 4.8 (3)]{dyer2024shi} we have $L^0 \subseteq L$ and since every super-elementary root is elementary (\cite[Proposition 7.2]{PARKINSON2022108146}) it follows by \Cref{thm:gates_characterised_by_phi0} that $\Gamma^0 \subseteq L^0 \subseteq L$. Note also that $L$ is finite (since $\mathscr{E}$ is finite).

In general, the set $\mathcal{S}$ can be determined in finite time by considering each root $\beta \in \mathscr{E}$ and checking whether there exists a pair of gates or low-elements $\{x, y\}$ with $\Phi(x) \cap \Phi(y) = \{ \beta \}$. An application of  \Cref{thm:witnesses_are_gated} and \Cref{suffix_closure_of_tight_gates} gives an alternative method which concurrently computes the set $\mathcal{S}$ and $\Gamma^0$ without first having to compute the entire set $\Gamma$ or $L$. It only requires that the elementary roots $\mathscr{E}$ is computed first (which is quite efficient, see \cite[Chapter 4.7]{BB05}; we note that elementary roots are also called \textit{small} roots in the literature). We remind the reader that $\Delta$ is the set of simple roots in the following algorithm.

\begin{algorithm}[H]
	\caption{Compute tight gates and super-elementary roots}
	\label{algo:compute_tight_gates_and_super_elementary_roots}
	\begin{algorithmic}[1]
		\STATE Let $L' = S$, $\mathcal{S'} = \Delta$, $\Gamma' = \emptyset$
		\STATE Initialize $L'_i = \{x \in L' \mid \ell(x) = i \}$ with $i = 1$
		\WHILE{$L'_i \neq \emptyset$}
		\FOR{$x \in L'_i$}
		\FOR{$s \notin D_L(x)$}
		\STATE Set $y = s \cdot x$
		\IF{$|\Phi^R(y)| = 1$ and $\{ \beta \} = \Phi^R(y) \in \mathscr{E}$}
		\STATE Add $y$ to $L'$
		\IF{there is $z \in L'$ with $\Phi^R(z) = \{ \beta \}$ and $\Phi(z) \cap \Phi(y) = \{ \beta \}$}
		\STATE Add $\beta$ to $\mathcal{S'}$
		\STATE Add $y, z$ to $\Gamma'$
		\ENDIF
		\ENDIF
		\ENDFOR
		\ENDFOR
		\STATE Set $i = i + 1$
		\ENDWHILE
	\end{algorithmic}
	The algorithm terminates in finite time and upon termination $\Gamma' = \Gamma^0$ and $\mathcal{S'} = \mathcal{S}$
\end{algorithm}
\begin{proof}
By \Cref{suffix_closure_of_tight_gates}, the set of tight gates $\Gamma^0$ is closed under suffix: if $x\in\Gamma^0$ and $s\in D_L(x)$, then $sx\in\Gamma^0$
and $\ell(sx)=\ell(x)-1$. Since $\Gamma^0$ is finite, it follows that every element $g\in\Gamma^0$ admits a length--decreasing chain contained in $\Gamma^0$
\begin{equation*}
    g = g_k \xrightarrow{s_k} g_{k-1} \xrightarrow{s_{k-1}} \cdots \xrightarrow{s_2} g_1,
\end{equation*}
where $g_1\in S$, $\ell(g_i)=i$, and $g_{i-1}=s_i g_i$ with $s_i\in D_L(g_i)$. Equivalently, every tight gate can be obtained from a simple reflection by successively left-multiplying by simple reflections $s\notin D_L(x)$, increasing length by one at each step. The algorithm hence performs a breadth--first search with respect to length, starting from the simple reflections and iterating over all possible left multiplications $x\mapsto sx$ with $s\notin D_L(x)$. Consequently, every element of $\Gamma^0$ is eventually encountered and added to $L' \subseteq L^0$. As discussed on the previous page, since $\mathcal{S} \subseteq \mathscr{E}$, by the characterisation of $\Gamma$ given in \Cref{thm:gates_characterised_by_phi0} we have 

\begin{equation*}
\Gamma^0 \subseteq L^0= \{ y \in W \mid |\Phi^R(y)| = 1 \text{ and } \Phi^R(y) \subseteq \mathscr{E} \} \subseteq L    
\end{equation*}

Since $L$ is finite, the algorithm terminates in finite time. Finally, once an element $y\in\Gamma^0$ is encountered, \Cref{thm:witnesses_are_gated}
guarantees the existence of a unique element $z\in\Gamma^0$ with
\begin{equation*}
    \Phi(y)\cap\Phi(z)=\{\beta\}=\Phi^R(y)=\Phi^R(z)
\end{equation*}
so the algorithm adds $\beta$ to $\mathcal{S'}$ and adds $y$ and $z$ to $\Gamma'$. Thus, upon termination, $\Gamma'=\Gamma^0$ and $\mathcal{S}'=\mathcal{S}$.
\end{proof}

We also note the following consequence of \Cref{thm:witnesses_are_gated} which shows that the number of non-simple tight gates is always even.

\begin{corollarySec}
    Let $K := |\Gamma^0 \setminus S|$. Then $K$ is even, and $|\mathcal{S} \setminus \Delta| \le K/2 $.
\end{corollarySec}
\begin{proof}
    \Cref{thm:witnesses_are_gated} shows that tight gates $g \in \Gamma^0 \setminus S$ are naturally paired according to their final roots $\Phi^R(g)$ in the following way: For each $\beta \in \mathcal{S} \setminus \Delta$, there is some $x \in \Gamma^0$ with $\beta \in \partial T(x^{-1}) \cap \Phi^R(x)$. Hence by \Cref{thm:witnesses_are_gated} there is a unique element $y \in \Gamma^0 \setminus S$ such that $\Phi(x) \cap \Phi(y) = \{ \beta \}$ and $\Phi^R(y) = \{ \beta \}$. By symmetry, $x$ is also the unique element of $\Gamma^0$ with $\Phi(x) \cap \Phi(y) = \{ \beta \}$.
\end{proof}

\subsection{Data for select Coxeter groups} \label{subsec:data_for_select_coxeter_systems}

We utilise \Cref{algo:compute_tight_gates_and_super_elementary_roots} to compute the size of $\mathcal{S}$ and $\Gamma^0$ for some Coxeter systems in low rank. We note that as the size of $\Gamma$ increases, in general, the ratio $|\Gamma^0|/|\Gamma|$ tends to decrease, making the computation of $\Gamma^0$ generally much more efficient to determine whether two elements have the same cone type.

\begin{figure}[H]
	\centering
	\begin{minipage}{0.45\linewidth}
		\centering
\begin{tabular}{|c|c|c|c|c|c|c|}
	\hline
	$W$ & $|\mathcal{E}|$ & $|\mathcal{S}|$ & $|L|$ & $|L^0|$ & $|\Gamma|$ & $|\Gamma^0|$ \\ \hline
	$\widetilde{A_2}$ & 6 & 6 & 16 & 9 & 16 & 9 \rule{0pt}{2.5ex} \\ \hline
	$\widetilde{B_2}$ & 8 & 8 & 25 & 14 & 24 & 13 \rule{0pt}{2.5ex} \\ \hline
	$\widetilde{G_2}$ & 12 & 12 & 49 & 26 & 41 & 21 \rule{0pt}{2.5ex} \\ \hline
	$\widetilde{A_3}$ & 12 & 12 & 125 & 28 & 125 & 28 \rule{0pt}{2.5ex} \\ \hline
	$\widetilde{B_3}$ & 18 & 18 & 343 & 66 & 315 & 58 \rule{0pt}{2.5ex} \\ \hline
	$\widetilde{C_3}$ & 18 & 18 & 343 & 66 & 317 & 58 \rule{0pt}{2.5ex} \\ \hline
	$\widetilde{A_4}$ & 20 & 20 & 1296 & 75 & 1296 & 75 \rule{0pt}{2.5ex} \\ \hline
	$\widetilde{B_4}$ & 32 & 32 & 6561 & 270 & 5789 & 227 \rule{0pt}{2.5ex} \\ \hline
	$\widetilde{C_4}$ & 32 & 32 & 6561 & 270 & 5860 & 227 \rule{0pt}{2.5ex} \\ \hline
	$\widetilde{D_4}$ & 24 & 24 & 2401 & 140 & 2400 & 139 \rule{0pt}{2.5ex} \\ \hline
	$\widetilde{F_4}$ & 48 & 48 & 28561 & 1054 & 22428 & 715 \rule{0pt}{2.5ex} \\ \hline
	$\widetilde{A_5}$ & 30 & 30 & 16807 & 186 & 16807 & 186 \rule{0pt}{2.5ex} \\ \hline
	$\widetilde{B_5}$ & 50 & 50 & 161051 & 1030 & 137147 & 836 \rule{0pt}{2.5ex} \\ \hline
	$\widetilde{C_5}$ & 50 & 50 & 161051 & 1030 & 139457 & 836 \rule{0pt}{2.5ex} \\ \hline
	$\widetilde{D_5}$ & 40 & 40 & 59049 & 608 & 58965 & 596 \rule{0pt}{2.5ex} \\ \hline
\end{tabular}
	\end{minipage}
	\hspace{0.05\linewidth}
	\begin{minipage}{0.45\linewidth}
		\centering
		\begin{tabular}{|c|c|c|c|c|c|c|}
			\hline
			$W$ & $|\mathcal{E}|$ & $|\mathcal{S}|$ & $|L|$ & $|L^0|$ & $|\Gamma|$ & $|\Gamma^0|$ \\ \hline
			$X_4(4)$ & 25 & 25 & 438 & 112 & 392 & 98 \\ \hline
			$X_4(5)$ & 32 & 32 & 516 & 158 & 462 & 138 \\ \hline
			$Y_4$ & 32 & 32 & 687 & 166 & 578 & 150 \\ \hline
			$Z_4$ & 30 & 30 & 513 & 142 & 473 & 132\\ \hline
			$X_5(3)$ & 114 & 114 & 101412 & 5767 & 52542 & 4071 \\ \hline
			$X_5(4)$ & 83 & 83 & 25708 & 3128 & 22886 & 2871 \\ \hline
			$X_5(5)$ & 135 & 135 & 42064 & 6014 & 37956 & 5523 \\ \hline
			$Z_5$ & 120 & 120 & 41385 & 5476 & 39138 & 5391 \\ \hline
		\end{tabular}
	\end{minipage}
	\caption{Data for low rank affine and compact hyperbolic Coxeter groups.}
	\label{table:data_for_tight_gates}
\end{figure}

\section{Canonical join-representations} \label{sec:canonical_join_representations}

Our exploration of tight gates and join-representations of elements in \Cref{sec:Tight gates and Join-irreducible elements of  Gamma} led us close to the formative work of Reading and Speyer on canonical join-representations of elements in Coxeter groups. In this complimentary section, we give an alternative proof of \cite[Theorem 8.1]{e2fecb84-5a76-3abb-b4fd-01261d3c05ed} by applying \Cref{prop:bijection_join_refine_phi1} and utilising the \textit{short inversion} poset recently developed by Dyer, Hohlweg, Fishel and Marks \cite{dyer2024shi} in their study of low-elements and Shi-arrangements (see \Cref{subsec:maximal_dihedral_and_short_inversion_poset}). We first recall the definition of the notion of a \textit{canonical join-representation}.

\begin{definition}
    Let $L$ be a poset and $x \in L$. An expression $x = \bigvee U$ is a \textit{canonical join representation} if $U$ is irredundant and if any other join representation $x = \bigvee V$ has $U << V$.
\end{definition}

The following theorem was first proven by Reading (\cite[Theorem 10-3.9]{GratzerWehrung2016}) for finite Coxeter groups in his seminal work on the Lattice Theory of the Poset of Regions of hyperplane arrangements, and then later generalised to all Coxeter groups with Speyer.

\begin{theorem} \cite[Theorem 8.1]{e2fecb84-5a76-3abb-b4fd-01261d3c05ed}
\label{thm:canonical_join_representation}
   Let $(W,S)$ be a finitely generated Coxeter group and $w \in W$. For each $\beta \in \Phi^R(w)$ there is a unique minimal length element $j_{\beta} \in \{ v \mid v \preceq w, \beta \in \Phi(v) \}$. The canonical join representation of $w$ is $w = \bigvee \{ j_{\beta} \mid \beta \in \Phi^R(w) \}$.
\end{theorem}

\subsubsection{Maximal dihedral reflection subgroups and the short inversion poset} \label{subsec:maximal_dihedral_and_short_inversion_poset}

Let us briefly recall the notion of maximal dihedral reflection subgroups in Coxeter groups as this is required for the definition of the short inversion poset (introduced in \cite{dyer2024shi}) which we utilise in the proof of \Cref{thm:canonical_join_representation}. For a more in-depth discussion see \cite[Section 2.5]{HD16}. A dihedral reflection subgroup $W'$ of $W$ is a subgroup of rank 2. Every dihedral reflection subgroup is contained in a \textit{maximal} dihedral reflection subgroup (note that dihedral subgroups may be infinite).

For $\alpha , \beta \in \Phi^+$, we denote $\mathcal{M}_{\alpha, \beta}$ to be the  \textit{maximal} dihedral reflection subgroup whose associated root system $\Phi_{\mathcal{M}_{\alpha, \beta}}$ contains $\alpha$ and $\beta$. The simple roots of $\Phi_{\mathcal{M}_{\alpha, \beta}}$ is denoted $\Delta_{\mathcal{M}_{\alpha, \beta}}$.

We can now define the short inversion poset. Let $w \in W$. For $\alpha, \beta \in \Phi^1(w)$ define $\alpha \dot{{\prec}} \beta$ if $\beta \notin \Delta_{\mathcal{M}_{\alpha, \beta}}$. This is equivalent to $\alpha \in \Delta_{\mathcal{M}_{\alpha, \beta}}$ and $\beta \notin \Delta_{\mathcal{M}_{\alpha, \beta}}$ (See \cite[Proposition 3.2]{dyer2024shi}). Then define the relation $\preceq_{w}$ as the transitive and reflexive closure of $\alpha \dot{{\prec}} \beta$. This is a partial order on $\Phi^1(w)$.

\begin{theoremSec} \cite[Theorem 3.17]{dyer2024shi} \label{thm:short_inversion_sandwich}
    Let $w \in W$. For the poset $(\Phi^1(w), \preceq_w)$, the minimal elements are the left-descent roots in $\Phi^L(w)$ and the maximal elements are the right-descent roots in $\Phi^R(w)$. In other words, for any $\beta \in \Phi^1(w)$ there is $\alpha \in \Phi^L(w)$ and $\gamma \in \Phi^R(w)$ such that $\alpha \preceq_w \beta \preceq_w \gamma$.
\end{theoremSec}

For our purposes, from \Cref{thm:short_inversion_sandwich} we utilise the fact that for $w \in W$ and any $\beta \in \Phi^1(w)$, there is a \textit{final} root $\gamma \in \Phi^R(w)$ such that $\beta \preceq_w \gamma$. In particular, the following fact (supplied by the next proposition) is critical for the proof of \Cref{thm:canonical_join_representation}: If $\beta \in \Phi^1(w)$ then there is $\gamma \in \Phi^R(w)$ such that for any reduced word of $w$, the wall $H_{\beta}$ is crossed before $H_{\gamma}$ in the geodesic path in $X^1$ from the identity $e$ to $w$.

\begin{proposition} \cite[Proposition 3.13]{dyer2024shi} \label{prop:short_inversion_poset_determines_order}
    The relation $\preceq_{w}$ is a partial order on $\Phi^1(w)$. Moreover, for any reduced word $w = s_1 s_2 \ldots s_k$ consider the following total order $\leq$ on $\Phi(w)$:
    \begin{equation*}
        \alpha_{s_{1}} < s_1 \alpha_{s_2} < \ldots < s_1 \ldots s_{k-1} \alpha_{s_{k}}
    \end{equation*}
    Then $\alpha \preceq_w \beta$ implies $\alpha \le \beta$ for any $\alpha, \beta \in \Phi^1(w)$.
\end{proposition}

We can now prove \Cref{thm:canonical_join_representation}.

\begin{proof}
    The first part of the theorem follows by \Cref{prop:bijection_join_refine_phi1} since $\Phi^R(w) \subseteq \Phi^1(w)$. For the second part, we first show that $w = \bigvee \{ j_{\beta} \mid \beta \in \Phi^R(w) \}$. First, note that $j_{\beta} \preceq w$ for all $\beta \in \Phi^R(w)$ and so $v := \bigvee \{ j_{\beta} \mid \beta \in \Phi^R(w) \}$ exists with $v \preceq w$. Our aim is to show that $v = w$. For $\alpha \in \Phi^1(w)$, let $j_{\alpha}$ be the unique minimal length element of $\{v \mid v \preceq w, \alpha \in \Phi(v) \}$. 
    
    By \Cref{inversion_set_of_join}, it suffices to show that for each $\alpha \in \Phi^1(w) \setminus \Phi^R(w)$ there is $\beta \in \Phi^R(w)$ such that $w_{\alpha} \preceq j_{\beta}$ for some $w_{\alpha} \preceq w$ with $\alpha \in \Phi^1(w_{\alpha})$; since then we have $\Phi^1(w) \subseteq \bigcup_{\beta \in \Phi^R(w)} \Phi(j_{\beta})$ and thus by \Cref{inversion_set_is_determined_by_phi1} and \Cref{inversion_set_of_join} it would then follow that 
    \begin{equation*}
        \Phi(w) = cone_{\Phi}(\Phi^1(w)) \subseteq cone_{\Phi} \big(\bigcup_{\beta \in \Phi^R(w)} \Phi(j_{\beta}) \big) = \Phi(v)
    \end{equation*}
    and so $w \preceq v$.
    
    Hence let $\alpha \in \Phi^1(w) \setminus \Phi^R(w)$. By \Cref{thm:short_inversion_sandwich} there is $\beta \in \Phi^R(w)$ such that $\alpha \preceq_w \beta$ and by \Cref{prop:short_inversion_poset_determines_order} this implies that for any reduced word $\bar{w}$ of $w$ we have $\alpha \le \beta$ in the total order on $\Phi(w)$. Intuitively, this means that the wall $H_{\alpha}$ is crossed \textit{before} the wall $H_{\beta}$ for any reduced word $\bar{w}$ of $w$, when considered as a walk from the identity $e$ to $w$ in the Coxeter complex or Cayley graph of $(W,S)$.
    
    Now let $\bar{w}$ be a reduced word of $w$ and let $\bar{w}_{\beta}$ be the initial path of $\bar{w}$ with $\beta \in \Phi^R(w_{\beta})$, where $w_{\beta}$ is the element represented by the word $\bar{w}_{\beta}$. Note that $\bar{w}_{\beta}$ is a reduced word since $\bar{w}$ is reduced. By \Cref{prop:bijection_join_refine_phi1}, since $\beta \in \Phi^1(w_{\beta})$ and $w_{\beta} \preceq w$, we have  $j_{\beta} \preceq w_{\beta}$ and so we can modify $\bar{w}$ to obtain another reduced word of $w$, $\bar{\bar{w}}$ with a reduced word $\bar{j_{\beta}}$ the initial path of $\bar{\bar{w}}$ with $\beta \in \Phi^R(j_{\beta})$. 
    
    Applying the same argument to the reduced word $\bar{\bar{w}}$, by \Cref{prop:short_inversion_poset_determines_order} again, there must be an initial path of $\bar{\bar{w}}$ with reduced word $\bar{w}_{\alpha}$ with $\alpha \in \Phi^R(w_{\alpha})$ such that $w_{\alpha} \preceq j_{\beta} \preceq w$ as required.
    
    Now let $\phi^R(w) = \{ j_{\beta} \mid \beta \in \Phi^R(w) \}$. We now show that $\phi^R(w)$ is irredundant. If $|\Phi^R(w)| = 1$, then $w$ is join-irreducible and the result is clear, hence assume $|\Phi^R(w)| > 1$. Suppose  for some $\alpha \in \Phi^R(w)$ we have
    \begin{align*}
        w =& \  \bigvee \{ j_{\beta} \mid \beta \in \Phi^R(w) \setminus \{ \alpha \} \} \\
        =& \ \bigvee \{ j_{\beta} \mid \beta \in \Phi^1(w) \setminus \{ \alpha \} \} \hspace{0.5cm} \mbox{by \Cref{prop:x_is_join_of_phi1_tight_elements}}
    \end{align*}
   It then follows that $s_{\alpha} w \prec w$ and $\Phi^1(s_{\alpha}w) = \Phi^1(w) \setminus \{ \alpha \}$ and thus
    \begin{equation*}
        w \succ s_{\alpha}w = \bigvee \{ j_{\beta} \mid \beta \in \Phi^1(w) \setminus \{ \alpha \} \} \hspace{0.5cm} \mbox{again by \Cref{prop:x_is_join_of_phi1_tight_elements}}
    \end{equation*}
    a contradiction. Finally, the fact that the set $\phi^R(w)$ join-refines every join representation of $w$ follows by \Cref{join_refinements} since $\phi^R(w) \subseteq \phi^1(w)$.
\end{proof}

\newpage

\bibliographystyle{plain}
\bibliography{journal}
\end{document}